\documentclass[11pt]{amsart}
\usepackage{amsthm,amsmath,amssymb}

\usepackage{geometry}
\usepackage{graphicx,cite}
\usepackage{color}

\numberwithin{equation}{section}
\numberwithin{figure}{section}
\theoremstyle{plain}
\newtheorem{thm}{Theorem}[section]

\newtheorem{claim}[thm]{Claim}

\theoremstyle{remark}
\newtheorem{rmk}[thm]{Remark}

\newcommand{\M}{\operatorname{M}}
\newcommand{\Hf}{\operatorname{\textbf{H}}}

\newcommand{\PP}{\operatorname{PP}}
\begin{document}

\title{A Shuffling Theorem  for Centrally Symmetric Tilings}

\author{Tri Lai}
\address{Department of Mathematics, University of Nebraska -- Lincoln, Lincoln, NE 68588}
\email{tlai3@unl.edu}
\thanks{This research was supported in part  by Simons Foundation Collaboration Grant (\# 585923).}



\date{\today}

\dedicatory{}

\begin{abstract}
Rohatgi and the author recently proved a shuffling theorem for lozenge tilings of `doubly--dented hexagons' (\texttt{arXiv:1905.08311}). The theorem can be considered as a hybrid between two classical theorems in the enumeration of tilings: MacMahon's theorem about centrally symmetric hexagons and Cohn--Larsen--Prop's theorem about semihexagons with dents. In this paper, we consider a similar shuffling theorem for the centrally symmetric tilings of the doubly--dented hexagons.  Our theorem also implies a conjecture posed by the author in \texttt{arXiv:1803.02792} about the enumeration of centrally symmetric tilings of hexagons with three arrays of triangular holes. This enumeration, in turn, can be considered as a common generalization of (a tiling-equivalent version of) Stanley's enumeration of self-complementary plane partitions and Ciucu's work on symmetries of the shamrock structure. Moreover, our enumeration also confirms a recent conjecture posed by Ciucu in \texttt{arXiv:1906.02951}.
\end{abstract}

\maketitle

\section{Introduction}

MacMahon's classical theorem  \cite{Mac} on plane partitions fitting in a given box is equivalent to the fact that the number of lozenge tilings of a centrally symmetric hexagon $Hex(a,b,c)$ of side-lengths $a,b,c,a,b,c$ (in the clockwise order from the north side\footnote{From now on we always list the side-lengths of a hexagon in the clockwise order, starting from the north side.}) is given by the simple product:
\begin{equation}\label{Maceq}
\PP(a,b,c):=\prod_{i=1}^{a}\prod_{j=1}^{b}\prod_{k=1}^{c}\frac{i+j+k-1}{i+j+k-2}.
\end{equation}

This formula was generalized by Cohn, Larsen and Propp \cite[Proposition 2.1]{CLP} when they presented a correspondence between lozenge tilings of a semihexagon with unit triangles removed on the base and semi-strict Gelfand--Tsetlin patterns. In particular, the \emph{(dented) semihexagon} $T_{a,b}(s_1,s_2,\dots,s_a)$ is the region obtained from the upper half of the symmetric hexagon of side-lengths $b,a,a,b,a,a$ by removing $a$ up-pointing unit triangles along the base at the positions $s_1,s_2,\dotsc,s_a$ from left to right (see Figure \ref{semihexmultiple} for an example).  The removed unit triangles are called the `\emph{dents}'. Cohn--Larsen--Propp's theorem says that the number of lozenge tilings of the dented semihexagon is given by
\begin{equation}\label{CLPeq}
\M(T_{a,b}(s_1,s_2,\dots,s_a))=\prod_{1\leq i<j \leq a}\frac{s_j-s_i}{j-i},
\end{equation}
where we use the notation $\M(R)$ for the number of lozenge tilings of the region $R$.

In \cite{shuffling}, Rohatgi and the author considered a hybrid object between MacMahon's hexagon and Cohn--Larsen--Propp's dented semihexagon. The region is a hexagon on the triangular lattice, like in the case of MacMahon's theorem,  with an arbitrary set of unit triangles removed along a horizontal axis, like the dents in Cohn--Larsen--Propp's theorem (illustrated in Figure \ref{multiplefernfig}). However, instead of removing only up-pointing unit triangles, we remove both up-pointing and down-pointing  unit triangles. We call this region a \emph{doubly-dented hexagon}. In general, the tiling number of such a region is not given by  a simple product formula. However, we showed that the tiling number only changes by a simple multiplicative factor when we shuffle the positions of up- and down-pointing removed unit triangles (see Theorem 1 in \cite{shuffling}).

For the completeness, let us present here the Shuffling Theorem in \cite[Theorem2.4]{shuffling}. Assume that $x,y,z,u,d,n$ are nonnegative integers, such that $u,d \leq n$. Consider a symmetric hexagon of side-lengths $x+n-u,y+u,y+d,x+n-d,y+d,y+u$\footnote{From now on, we always list the side-lengths of a hexagon in the clockwise order from the north side.}. We remove $u+d$ arbitrary unit triangles along the horizontal lattice line $l$ that contains the west and the east vertices of the hexagon. We call $l$ the \emph{axis} of the region. Assume further that, among these $u+d$ removed triangles, there are $u$ up-pointing ones and $d$ down-pointing ones. Denote respectively by $U=\{s_1,s_2,\dotsc,s_u\}$ and $D=\{t_1,t_2,\dots, t_d\}$ the sets of positions of the up-pointing and down-pointing removed unit triangles (ordered from left to right), such that $|U\cup D|=n$ (i.e., $U,D\subseteq [x+y+n]:=\{1,2,\dots,x+y+n\}$, $U$ and $D$ are \emph{not} necessarily disjoint).  We also allow the appearance of ``\emph{barriers}" at the positions in a set $B\subseteq [x+y+n]\setminus(U\cup D)$ along $\ell$, such that $|B|\leq x$ (a barrier is a unit horizontal lattice interval that is not allowed to be contained in a vertical lozenge of any tilings; see the red barriers in Fig.  \ref{multiplefernfig}; $B=\{6,13\}$ in this case). It means that we do not allow the appearance of vertical lozenges at the positions in $B$.  We use the notation $H_{x,y}(U;D;B)$ for the hexagon with the above setup of removed unit triangles and barriers. We call the resulting region a \emph{doubly--dented hexagon (with barriers)}. See Figure  \ref{multiplefernfig} for an example of the doubly-dented hexagon and a sample tiling of its.

We now `shuffle' and/or `flip' the up- and down-pointing unit triangles in the symmetric difference $U\Delta D$ to obtain new position sets $U'=\{s'_1,s'_2,\dots,s'_{u'}\}$ and $D'=\{t'_1,t'_2,\dots,t'_{d'}\}$ for the up-pointing removed triangles and the down-pointing removed triangles. In particular, $U\cup D =U'\cup D'$ and $U\cap D =U'\cap D'$ ($U$ and $U'$ and $D$ and $D'$ may have different sizes). The following theorem shows that the shuffling and flipping of removed triangles only changes the tiling number  by a simple multiplicative factor. Moreover, the factor can be written in a similar form to Cohn--Larsen--Propp's formula (i.e. the product on the right-hand side of equation \ref{CLPeq}).

\begin{figure}\centering
\includegraphics[width=13cm]{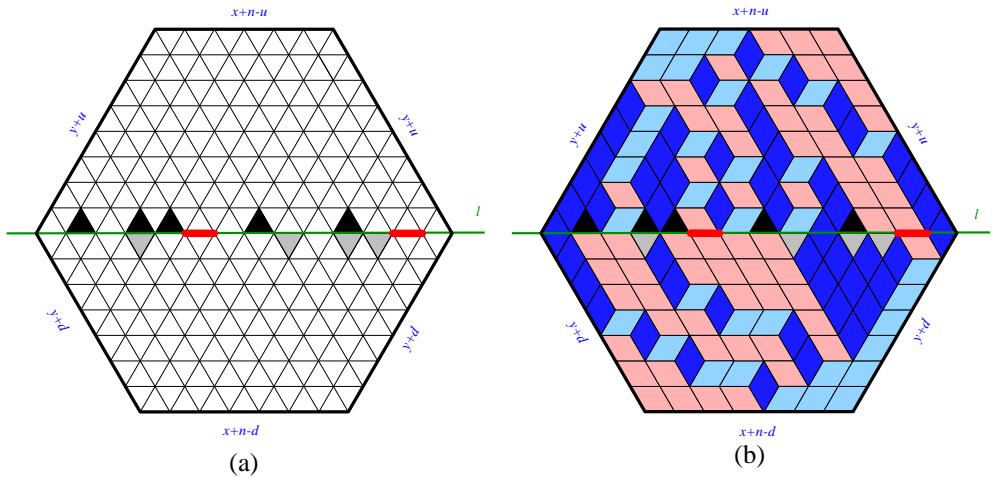}
\caption{(a) The doubly--dented hexagon $H_{4,3}(2,4,5,8,11;\ 4,9,11,12;\ 6,13)$ and (b) a lozenge tiling of its. The black and shaded triangles indicate the unit triangles removed.}\label{multiplefernfig}
\end{figure}

\begin{thm}[Shuffling Theorem, Theorem 2.4 in \cite{shuffling}]\label{factorization} For nonnegative integers $x,y,n,u,d,u',d'$  ($u,d,u',d'\leq n$) and five ordered subsets $U=\{s_1,s_2,\dotsc,s_u\}$,  $D=\{t_1,t_2,\dots, t_d\}$, $U'=\{s'_1,s'_2,\dotsc,s'_{u'}\}$,  and $B=\{k_1,k_2,\dots,k_b\}$ of $[x+y+n]$,  such that $U\cup D =U'\cup D'$, $U\cap D =U'\cap D'$, $B\cap (U\cup D)=\emptyset$, and $b=|B|\leq x$. Then

\begin{equation}\label{genmaineq2}
  \frac{\M(H_{x,y}(U;D;B))}{\M(H_{x,y}(U';D';B))}= \frac{\displaystyle\prod_{1\leq i <j\leq u}\frac{s_j-s_i}{j-i}\displaystyle\prod_{1\leq i <j\leq d}\frac{t_j-t_i}{j-i}\PP(u,d,y)}{\displaystyle\prod_{1\leq i <j\leq u'}\frac{s'_j-s'_i}{j-i}\displaystyle\prod_{1\leq i <j\leq d'}\frac{t'_j-t'_i}{j-i}\PP(u'd',y)}.
\end{equation}
\end{thm}

Motivated by Stanley's classical paper \cite{Stanley} on symmetric plane partitions, one would like to investigate symmetric tilings of  the doubly-dented hexagons. There are two natural classes of symmetric tilings: reflectively symmetric tilings (the tilings which are invariant under a refection over the symmetry axis) and  centrally symmetric tilings (the tilings which are invariant under a $180^{\circ}$ rotation around the symmetry center)\footnote{These types of symmetric tilings correspond to the transposed-complementary and self-complimentary plane partitions, respectively.}. While a shuffling theorem for the first symmetry class is considered in \cite{shuffling2}, in this paper, we investigate shuffling theorems for the second symmetry class.

\begin{figure}\centering
\includegraphics[width=13cm]{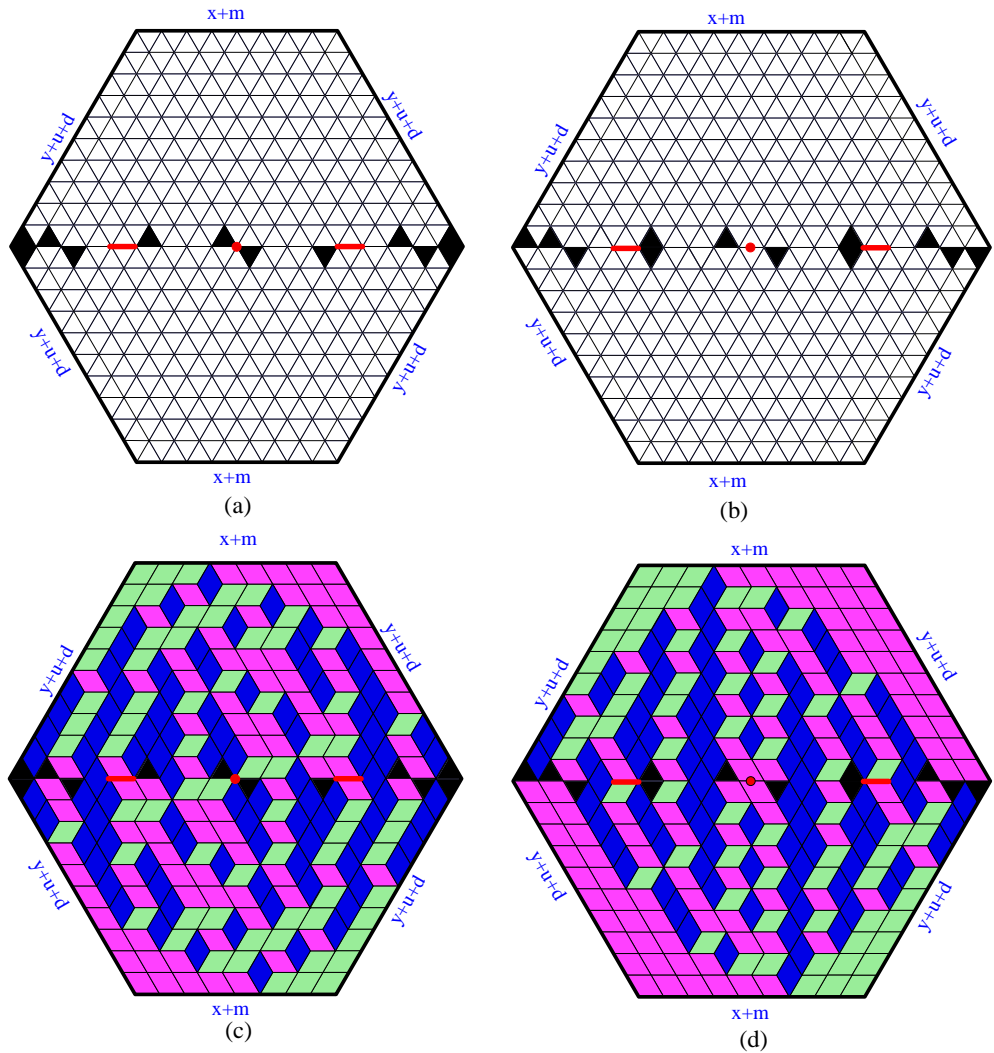}
\caption{(a) The centrally symmetric doubly-dented hexagon $CS_{2,2}(1,2,6,9;\ 1,3;\ 5)$; the dot indicates the center of the region. (b) The centrally symmetric doubly-dented hexagon $CS_{3,2}(1,2,6,9;\ 3,6;\ 5)$. (c) A centrally symmetric tiling of the region in picture (a). (d) A centrally symmetric tiling of the region in picture (b).}\label{Symmetrictiling}
\end{figure}

\medskip

We now consider a centrally symmetric doubly--dented hexagon with a set of barriers as follows. Consider the symmetric hexagon of side-lengths $x+2n-u-d,y+u+d,y+u+d,x+2n-u-d,y+u+d,y+u+d$. We are removing $u+d$ up-pointing triangles and $u+d$ down-pointing triangles that are located symmetrically at $2n$ positions along the horizontal axis $l$. We are also placing $2b$ barriers symmetrically along $l$.  The index sets of removed up-pointing triangles, removed down-pointing triangles, and barriers are respectively $U\cup((x+y+2n+1)-D)$, $D\cup ((x+y+2n+1)-U)$, $B\cup ((x+y+2n+1)-B)$, for some index sets $U=\{s_1<s_2<\cdots<s_u\}$, $D=\{t_1<t_2<\cdots<t_d\}$, $B=\{k_1<k_2<\cdots<k_b\}\subseteq \{1,2,\dots, \left\lceil \frac{x+y+n}{2}\right\rceil\}$, such that $B\cap (U\cup D)=\emptyset$ and $2b\leq x$. Here, for an index set $W$, we denote $(x+y+2n+1)-W:=\{x+y+2n+1-w:\ w \in W\}$.  Denote by $CS_{x,y}(U; D;B)$ the (centrally symmetric) doubly-dented hexagon
\[H_{x,y}(U\cup((x+y+2n+1)-D);\  D\cup ((x+y+2n+1)-U);\  B\cup ((x+y+2n+1)-B)).\]
We see that the set of removed unit triangles and the set of barrier in $CS_{x,y}(U; D;B)$ are invariant under the central symmetry. 
By the symmetry, one readily sees that $2n-u-d=m:=|U\Delta D|$.  Figures \ref{Symmetrictiling}(a) and (b) show respectively a $CS$-type region in the case when the symmetry center is a lattice point and in the case the symmetry center is not a lattice point. Figures \ref{Symmetrictiling}(c) and (c) show a centrally symmetric tiling of the $CS$-type regions in Figures \ref{Symmetrictiling}(a) and (b), respectively.

We have the following shuffling theorem for centrally symmetric tilings of the $CS$-type regions.
\begin{thm}[Shuffling Theorem for Centrally Symmetric Tilings]\label{centralfactor} Assume that $x,y,n,u,d,u',d'$ $U=\{s_1<s_2<\cdots<s_u\}$ are nonnegative integers and that $D=\{t_1<t_2<\cdots<t_d\}$, $B=\{k_1<k_2<\cdots<k_b\}$, $U'= \{s'_1<s'_2<\cdots<s'_{u'}\}$, $D'=\{t'_1<t'_2<\cdots<t'_{d'}\}$, $B=\{k_1<k_2<\cdots<k_b\}$ are five oredered subsets of $\{1,2,\dots,\left\lceil \frac{x+y+n}{2}\right\rceil\}$, such that $U\cup D=U'\cup D'$, $U\cap D=U'\cap D'$, $B\cap (U\cup D)=\emptyset$, and $2b\leq x$. We have
\begin{align}\label{genmaineq3}
  \frac{\M_c(CS_{x,y}(U;D;B))}{\M_c(CS_{x,y}(U';D';B))}&=\sqrt{\frac{\M(CS_{x,y}(U;D;B))}{\M(CS_{x,y}(U';D';B))}}\notag\\
  &=\frac{\Delta(U\cup((x+y+2n+1)-D))}{\Delta(U'\cup ((x+y+2n+1)-D'))},
\end{align}
where $\M_c(R)$ denotes the number of centrally symmetric tilings of region $R$ and where, for a finite ordered set $S=\{x_1<x_2<\cdots<x_n\}$, we define $\Delta(S):=\prod_{1\leq i<j\leq n}(x_j-x_i)$.
\end{thm}

Our theorem also implies an exact enumeration of centrally symmetric tilings of hexagons with three ferns removed, which confirms the author's conjecture in \cite{Threefern}.  As a hexagon with three ferns removed is a common generalization of a centrally symmetric hexagon $Hex(x,y,z)$ and the centrally symmetric hexagon with two triangles removed from the center (the latter is denoted by $B_{x,y,z,k}$ or $B'_{x,y,z,k}$ in \cite{Ciu3}; see Figures \ref{figsymfern2}(b) and \ref{figsymfern3}(b) for examples), our theorem in turn can be considered as a common generalization of Stanley's enumeration of self-complementary plane partitions (equivalent, centrally symmetric tilings of  a hexagon) \cite{Stanley} and Ciucu's work on the symmetries of the so-called `\emph{shamrock}' in \cite{Ciu3}. We refer the reader to e.g. \cite{Andrews, Kup, Stem, Krat3, Zeil} and the lists of references therein for more work about symmetric plane partitions. After posting the initial version of this paper on \texttt{arXiv.org}, the author has learned that our main theorem also gives a confirmation for a recent conjecture of Ciucu (Conjecture 4 in \cite{Ciu4}).

The rest of the paper is organized as follows. In Section 2, we consider several applications of our main theorem to the new type of holes, called `\emph{fern}',  that has been studied recently by Ciucu and by the author.  In Section 3, we present a special version of Kuo condensation developed by Ciucu in \cite{Ciu3}. This result helps us create recurrences in our inductive proofs of the main theorem in Section 4. 

\section{Enumeration of hexagons with three ferns removed}

In \cite{Threefern}, the author gives exact enumeration for lozenge tilings of  hexagons with three chains of triangles with alternative orientations removed. This chain is called a `\emph{fern}'. The fern structure has been investigated recently by Ciucu and the author (see e.g. \cite{Ciu1, Twofern, Threefern, Threefern2, Halfhex1, Halfhex2, Halfhex3}\footnote{In these papers, a fern was encoded as a sequence of side-lengths of the triangles that it contains.}). The author also conjectures the existence of a simple product formula for the number of centrally symmetric tilings of a hexagons with three ferns removed (Conjecture 4.1 in \cite{Threefern}). In this section, we show that our main theorem (Theorem \ref{centralfactor}) implies the conjecture.

In the rest of this section, we restrict ourself in the special case when the up-index set $U$ and down-index set $D$ of the doubly-dented hexagon $H_{x,y}(U;D;B)$ are disjoint and there are no barriers, i.e. $U\cap D=\emptyset$ and $B=\emptyset$. We now assume that the set of removed unit triangles in the doubly-dented hexagon is partitioned into $k$ disjoint `\emph{clusters}' (i.e. chains of contiguous removed unit triangles).  Denote by $C_1,C_2,\dotsc,C_k$ these clusters and the distances between them are $d_1,d_2,\dotsc,d_{k-1}$ ($d_i>0$, for $i=1,2,\dots,k-1$) as they appear from left to right.  For the sake of convenience, we always assume that $C_1$ is attaching to the west vertex of the hexagon, that $C_k$ is attaching to the east vertex of the hexagon, and that $C_1$ and $C_k$ may be empty.   As we are assuming $U\cap D=\emptyset$, each cluster can be partitioned into maximal sequences of unit triangles of the same orientation. We call these sequences \emph{up-intervals} or \emph{down-intervals} based on the orientation of the triangles that they contain. In each cluster $C_i$, by removing forced vertical lozenges right above each up-interval and right below each down-interval whenever they contain more than one unit triangles, we obtain a chain $F_i$ of triangular holes of alternating orientations, i.e a `fern'. We use the notation $R_{x,y}(F_1,\dots,F_k|\  d_1,\dots,d_{k-1})$ for the hexagon with the $k$ ferns $F_i$'s removed (see Fig. \ref{fern} for an example; the black unit triangles indicate the unit triangles removed).

\begin{figure}\centering
\includegraphics[width=8cm]{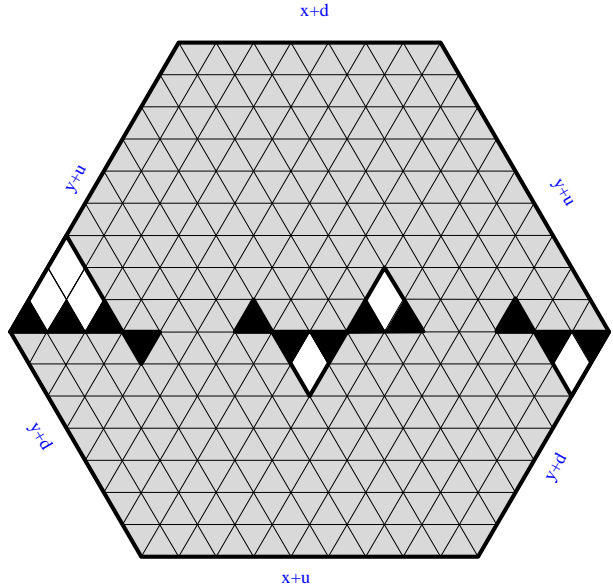}
\caption{Obtaining a hexagon with ferns removed $R_{x,y}(F_1,\dots,F_k|\  d_1,\dots,d_{k-1})$ from a doubly-dented hexagon.}\label{fern}
\end{figure}

We now shuffle and flip unit triangles \emph{internally} for each cluster $C_i$ to obtain a new cluster $C'_i$, for $i=1,2,\dots,k$. The removal of forced lozenges along the new clusters $C'_i$'s yields  new ferns $F'_i$'s. This way we get a new hexagon with ferns removed $R_{x,y}(F'_1,\dots,F'_k|\  d_1,\dots,d_{k-1})$. One notes that the only constrain here between the two ferns $F_i$ and $F'_i$ is that they have the same  \emph{total length} (the total length of a fern is the sum of side-lengths of triangles in the fern).

One can define the region $R_{x,y}(F_1,\dots,F_k|\  d_1,\dots,d_{k-1})$ directly as follows. We start with a hexagon of side-lengths $x+d,y+u,y+d,x+u,y+d,y+u$, where $u,d$ are the sum of side-lengths of the up-pointing and down-pointing triangles in the $k$ ferns $F_i$'s. We then remove the $k$ ferns $F_1,F_2,\dots,F_k$ from the hexagon along the horizontal axis $l$ containing its west and east vertices, such that the distances between two consecutive ferns are $d_1,d_2,\dots,d_{k-1}$ as appear from left to right, and that $F_1$ is attaching to the west vertex of the hexagon and $F_k$ is attaching to the east vertex.

For a given sequence, $\textbf{a}:=(a_1,a_2,\dots, a_n)$, we set
\[o_a:=a_1+a_3+a_5+\cdots.\]
\[e_a:=a_2+a_4+a_6+\cdots,\]
and
\[a:=a_1+a_2+a_3+\cdots.\]
Next, we define the \emph{generalized dented semihexagon} $S(a_1,a_2,\dotsc,a_n)$ as follows. We start with a trapezoidal region on the triangular lattice of side-lengths $e_a, o_a, a, o_a$ (in clockwise order, from the north side). We remove from the base of the region triangles of side-length $a_{2i-1}$'s such that the first one is touching the west vertex of the region and the distances between two consecutive removed triangles are $a_{2i}$'s (see the shaded region in Figure \ref{semihexmultiple} for an example). We also call these removed triangles \emph{dents}. One notes that the structures of the dents in $S(a_1,a_2,\dotsc,a_n)$ determine all four sides of the semihexagon.

Each generalized dented semihexagon $S(a_1,a_2,\dots,a_m)$ is obtained from the (original) dented semihexagon
\[T_{o_a,e_a}\left( \bigcup_{i\geq 1} \left[\sum_{j=1}^{2i-1}a_j+1,\sum_{j=1}^{2i}a_j\right] \right)\]
 in Cohn--Larsen--Propp's formula by removing  forced lozenges as in Figure \ref{semihexmultiple}, where $[a,b]:=\{a,a+1,\dots,b-1,b\}$ for integers $a\leq b$. Let $s(a_1,a_2,\dotsc,a_n)$ denote the number of lozenge tilings of $S(a_1,a_2,\dotsc,a_n)$. From Cohn--Larsen--Propp's formula (\ref{CLPeq}), we get
\begin{align}\label{semieq}
s(a_1,a_2,\dots,a_{2l-1})&=s(a_1,a_2,\dots,a_{2l})\notag\\
&=\dfrac{1}{\Hf(a_1+a_{3}+a_{5}+\cdots+a_{2l-1})}\notag\\
&\,\,\,\times\dfrac{\prod_{\substack{1\leq i\leq j\leq 2l-1,\,\text{$j-i+1$ odd}}}\Hf(a_i+a_{i+1}+\cdots+a_{j})}{\prod_{\substack{1\leq i\leq j\leq 2l-1,\,\text{$j-i+1$ even}}}\Hf(a_i+a_{i+1}+\cdots+a_{j})},
\end{align}
where the `\emph{hyperfactorial function}' $\Hf(n)$ is defined as $\Hf(n):=0!1!2!\cdots(n-1)!$.
The first equality holds due to forced lozenges in the tilings of $S(a_1,a_2,\dotsc,a_{2l})$, after whose removal one is left precisely with the region $S(a_1,a_2,\dotsc,a_{2l-1})$.

\begin{figure}\centering
\setlength{\unitlength}{3947sp}%
\begingroup\makeatletter\ifx\SetFigFont\undefined%
\gdef\SetFigFont#1#2#3#4#5{%
  \reset@font\fontsize{#1}{#2pt}%
  \fontfamily{#3}\fontseries{#4}\fontshape{#5}%
  \selectfont}%
\fi\endgroup%
\resizebox{8cm}{!}{
\begin{picture}(0,0)%
\includegraphics{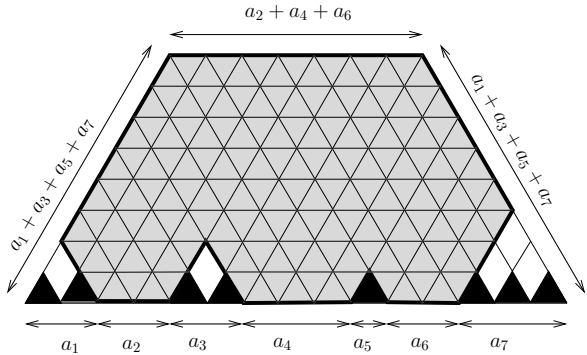}%
\end{picture}%
%
%

\begin{picture}(6833,4101)(1915,-5731)
\put(4741,-1911){\makebox(0,0)[lb]{\smash{{\SetFigFont{14}{16.8}{\rmdefault}{\mddefault}{\updefault}{\color[rgb]{0,0,0}$a_2+a_4+a_6$}%
}}}}
\put(7303,-2601){\rotatebox{300.0}{\makebox(0,0)[lb]{\smash{{\SetFigFont{14}{16.8}{\rmdefault}{\mddefault}{\updefault}{\color[rgb]{0,0,0}$a_1+a_3+a_5+a_7$}%
}}}}}
\put(2161,-4601){\rotatebox{60.0}{\makebox(0,0)[lb]{\smash{{\SetFigFont{14}{16.8}{\rmdefault}{\mddefault}{\updefault}{\color[rgb]{0,0,0}$a_1+a_3+a_5+a_7$}%
}}}}}
\put(2660,-5716){\makebox(0,0)[lb]{\smash{{\SetFigFont{14}{16.8}{\rmdefault}{\mddefault}{\updefault}{\color[rgb]{0,0,0}$a_1$}%
}}}}
\put(3351,-5711){\makebox(0,0)[lb]{\smash{{\SetFigFont{14}{16.8}{\rmdefault}{\mddefault}{\updefault}{\color[rgb]{0,0,0}$a_2$}%
}}}}
\put(4101,-5681){\makebox(0,0)[lb]{\smash{{\SetFigFont{14}{16.8}{\rmdefault}{\mddefault}{\updefault}{\color[rgb]{0,0,0}$a_3$}%
}}}}
\put(5061,-5681){\makebox(0,0)[lb]{\smash{{\SetFigFont{14}{16.8}{\rmdefault}{\mddefault}{\updefault}{\color[rgb]{0,0,0}$a_4$}%
}}}}
\put(5971,-5691){\makebox(0,0)[lb]{\smash{{\SetFigFont{14}{16.8}{\rmdefault}{\mddefault}{\updefault}{\color[rgb]{0,0,0}$a_5$}%
}}}}
\put(6621,-5681){\makebox(0,0)[lb]{\smash{{\SetFigFont{14}{16.8}{\rmdefault}{\mddefault}{\updefault}{\color[rgb]{0,0,0}$a_6$}%
}}}}
\put(7521,-5681){\makebox(0,0)[lb]{\smash{{\SetFigFont{14}{16.8}{\rmdefault}{\mddefault}{\updefault}{\color[rgb]{0,0,0}$a_7$}%
}}}}
\end{picture}%
}
\caption{Obtaining the region $S(2,2,2,3,1,2,4)$ (the shaded region with the bold contour) from the region $T_{7,8}(1,2,5,6,10,13,14,15)$ by removing several vertical forced lozenges; the black triangles indicate the unit triangle removed in the region $T_{7,8}(1,2,5,6,10,13,14,15)$.}\label{semihexmultiple}
\end{figure}

 Our Shuffling Theorem \ref{factorization} implies  the following fern-shuffling theorem.

\begin{thm}[Fern-shuffling Theorem]\label{fernthm} For nonnegative integer $x,y$ and positive integers $d_1,d_2,\dots,d_{k-1}$
\begin{equation}
\frac{\M(R_{x,y}(F_1,\dots,F_k|\ d_1,\dots,d_{k-1}))}{\M(R_{x,y}(F'_1,\dots,F'_k|\ d_1,\dots,d_{k-1}))}=\frac{\M(S^+)\M(S^{-})}{\M(S'^+)\M(S'^-)}\frac{\PP(u,d,y)}{\PP(u',d',y)},
\end{equation}
where $S^+$ and $S^-$ are the generalized dented semihexagons determined by the sequences of dents occurring above and below the horizontal axis of the ferns $F_i$ in $R_{x,y}(F_1,\dots,F_k|\ d_1,\dots,d_{k-1})$, respectively, where $u,d$ are the sum of side-lengths of the up-pointing and down-pointing triangles in the $k$ ferns $F_i$'s, and where $S'^+,S'^-, u', d'$ are defined similarly w.r.t. $R_{x,y}(F'_1,\dots,F'_k|\ d_1,\dots,d_{k-1})$.
\end{thm}

Next, we consider two fern-versions of our Shuffling Theorem for centrally symmetric tilings (i.e. Theorem \ref{centralfactor}) as follows.

We first consider the case when $x+y$ is even. Now assume that $F_1,\dots,F_k$ are $k$ ferns, in which $F_1$ and $F_k$ may be empty and $F_1,F_k$ are touching the west vertex and the symmetry center of the region. Denote by $\overline{F}_i$  is the fern obtained from $F_i$ by reflecting over the symmetry center.  Denote by $E_{x,y}(F_1,\dots,F_k|\ d_1,\dots,d_{k-1})$ the symmetric hexagon with $2k$ collinear ferns removed
\[R_{x,y}(F_1,\dots,F_k,\overline{F}_k,\dots,\overline{F}_1|\ d_1,\dots,d_{k-1},d_{k-1},\dots,d_1).\]
We note that in this case the symmetric center of the region is a lattice point. See Figure \ref{figsymfern2}(a) for an example.
It is easy to see that $x$ and $y$ have to be in fact even in order for the region has centrally symmetric tilings. Our Theorem \ref{centralfactor} implies that:

\begin{thm}\label{symmetricfernthm} Assume that $x,y$ are even nonnegative integers. Then
\begin{align}
\frac{\M_c(E_{x,y}(F_1,\dots,F_k|\ d_1,\dots,d_{k-1}))}{\M_c(E_{x,y}(F'_1,\dots,F'_k|\  d_1,\dots,d_{k-1}))}&=\sqrt{\frac{\M(E_{x,y}(F_1,\dots,F_k|\ d_1,\dots,d_{k-1}))}{\M(E_{x,y}(F'_1,\dots,F'_k|\ d_1,\dots,d_{k-1}))}}\notag\\
&=\frac{\M(S^+)}{\M(S'^+)}
\end{align}
where $S^+$ is the generalized dented semihexagon determined by the sequences of dents occurring above the axis of the ferns  in $E_{x,y}(F_1,\dots,F_k| d_1,\dots,d_{k-1})$ and where $S'^+$ is defined similarly w.r.t. $E_{x,y}(F'_1,\dots,F'_k| d_1,\dots,d_{k-1})$.
\end{thm}

We now assume $x+y$ is odd. We define a variation $E'_{x,y}(F_1,\dots,F_k|\ d_1,\dots,d_{k-1})$ of the region $E_{x,y}(F_1,\dots,F_k|\ d_1,\dots,d_{k-1})$ as follows.  $E'_{x,y}(F_1,\dots,F_k|\ d_1,\dots,d_{k-1})$ is the symmetric hexagon with $2k$ collinear ferns removed
\[R_{x,y}(F_1,\dots,F_k,\overline{F}_k,\dots,\overline{F}_1|\ d_1,\dots,d_{k-1},1,d_{k-1},\dots,d_1),\]
 such that the ferns $F_1$ and $F_k$ are touching the west vertex and the lattice point $1/2$ unit to the left of the symmetry center of the region (the symmetry center is now a middle of a unit horizontal lattice interval); $F_1,F_k$ may be empty. See Figure \ref{symfern3}(a) for an example.

  Our Theorem \ref{centralfactor} also  implies that:

\begin{thm}\label{symmetricfernthm2} For  nonnegative integers $x,y$ of opposite parities
\begin{align}
\frac{\M_c(E'_{x,y}(F_1,\dots,F_k| d_1,\dots,d_{k-1}))}{\M_c(E'_{x,y}(F'_1,\dots,F'_k| d_1,\dots,d_{k-1}))}&=\sqrt{\frac{\M(E'_{x,y}(F_1,\dots,F_k; d_1,\dots,d_{k-1}))}{\M(E'_{x,y}(F'_1,\dots,F'_k; d_1,\dots,d_{k-1}))}}\notag\\
&=\frac{\M(S^+)}{\M(S'^+)}
\end{align}
where $S^+$ is the generalized dented semihexagon determined by the sequences of dents occurring above the axis of the ferns  in $E'_{x,y}(F_1,\dots,F_k| d_1,\dots,d_{k-1})$ and where $S'^+$ is defined similarly w.r.t. $E'_{x,y}(F'_1,\dots,F'_k| d_1,\dots,d_{k-1})$.
\end{thm}

The author in \cite[Conjecture 4.1.]{Threefern} conjectured that the number of centrally symmetric tilings of a hexagon with three ferns removed is given by a simple product formula. By Theorem \ref{symmetricfernthm} and Ciucu's enumerations of the centrally symmetric tilings of the regions $B_{x,y,z,k}$ and $B'_{x,y,z,k}$ in \cite[Theorems 4 and 5]{Ciu3},  we have the following explicit formula for the latter number of symmetric tilings, which in turn confirms the author's conjecture.

\begin{figure}\centering
\includegraphics[width=15cm]{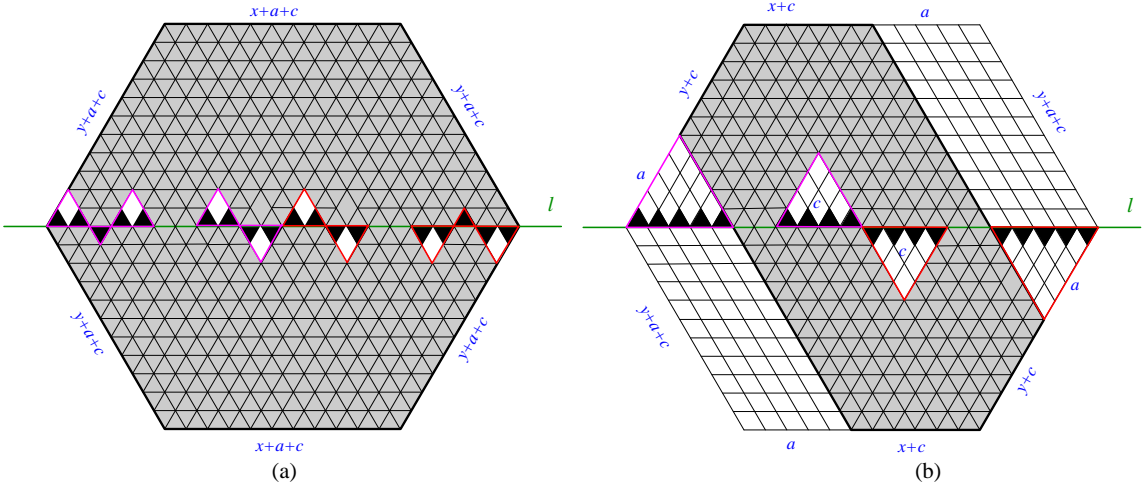}
\caption{Illustrating the proof of Theorem \ref{symmetricthreefern}. Here $a=a_1+a_2+a_3+\cdots$ and $c=c_1+c_2+c_3+\cdots$.}\label{figsymfern2}
\end{figure}

\begin{thm}\label{symmetricthreefern} For even nonnegative integers $x,y$ and two ferns $F_1$ and $F_2$ consisting of triangles of side-lengths $a_1,a_2,\dots,a_{m}$ and $c_1,c_2,\dots,c_k$, from left to right, respectively. Assume in addition that $m$ is odd (and the case of even $m$ can be obtained from the odd case by appending a triangle of side-length $0$ to the end of the $a$-fern). Then
\begin{align}\label{symfern1}
\M_c\left(E_{x,y}\left(F_1,F_2| \frac{x+y}{2}\right)\right)&=\frac{s\left(a_1,\dots,a_{m},\frac{x+y}{2},c_1,\dots,c_k,c_k,\dots, c_{1}+\frac{x+y}{2}+a_{m},a_{m-1},\dots,a_1\right)}{\PP\left(a,\frac{x+y}{2},c\right)}\notag\\
&\times \M_c(B_{x+c,y+2a+c,y+c,c})\notag\\
&=\frac{s\left(a_1,\dots,a_{m},\frac{x+y}{2},c_1,\dots,c_{2k}+\frac{x+y}{2}+a_{m}, a_{m-1},\dots,a_1\right)}{\PP\left(a,\frac{x+y}{2},c\right)} \notag\\
&\times\PP \left(\frac{y+2a+2c}{2},\frac{y}{2},c\right)\prod_{i=1}^{\frac{y}{2}+a}\frac{(\frac{x}{2}+i)_{c}}{(i)_{c}}\prod_{i=1}^{\frac{y}{2}}\frac{(\frac{x+2c}{2}+i)_{\frac{y}{2}+a}}{(c+i)_{\frac{y}{2}+a}}\frac{(\frac{x+2c}{2}+i)_{\frac{y+2a+2c}{2}}}{(c+i)_{\frac{y+2a+2c}{2}}},
\end{align}
where $a=a_1+a_2+a_3+\cdots$ and $c=c_1+c_2+c_3+\cdots$.
\end{thm}
We also note that the region $E_{x,y}(F_1,F_2|\frac{x+y}{2})$ is exactly the centrally symmetric hexagon with three ferns removed in Conjecture 4.1 in \cite{Threefern}. It means that Theorem \ref{symmetricthreefern} implies Conjecture 4.1 in \cite{Threefern}.

\begin{proof}
Assume that the fern $F_1$ consists of $m$ triangles of side-lengths $a_1,a_2,\dots,a_m$ from left to right, starting by an up-pointing triangles, the fern $F_2$ consists of $k$ triangles of side-lengths $c_1,c_2,\dots,c_k$ from left to right, starting by an up-pointing triangles. We pick $F'_1$ and $F'_2$ the ferns consisting of a single up-pointing triangle of side-lengths $a$ and $c$, respectively. Applying Theorem \ref{symmetricfernthm} to the region
\[E_{x,y}\left(F_1,F_2\ |\  \frac{x+y}{2}\right)=R_{x,y}\left(F_1,F_2\cup \overline{F}_2, \overline{F_1}\ |\  \frac{x+y}{2},\frac{x+y}{2}\right)\]
and the region
\[E_{x,y}(F'_1,F'_2\ |\  \frac{x+y}{2})=R_{x,y}\left(F'_1,F'_2\cup\overline{F'_2}, \overline{F'_1}\ |\  \frac{x+y}{2},\frac{x+y}{2}\right),\]
 we have the number of the centrally symmetric tilings of the first region is given by the number of centrally symmetric tilings of the second region, times a simple factor.

We note that the numbers of tilings of the dented semihexagons $S^+$ and $S'^+$ in Theorem \ref{symmetricfernthm} are given respectively by the $s$-functions in the numerator and denominator of the fraction after the first equality in (\ref{symfern1}). Moreover, after removing forced lozenges as in Figure \ref{figsymfern2}, the region $E_{x,y}(F'_1,F'_2| \frac{x+y}{2})$  becomes the region $B_{x+c,y+2a+c,y+c,c}$ in  \cite[Theorem 4]{Ciu3}. Then the first equality in (\ref{symfern1}) follows.  The second equality (\ref{symfern1}) follows from Ciucu's enumeration of centrally symmetric tilings of the region $B_{x,y,z,k}$ given in \cite[Theorem 4]{Ciu3}.
\end{proof}

\begin{figure}\centering
\includegraphics[width=15cm]{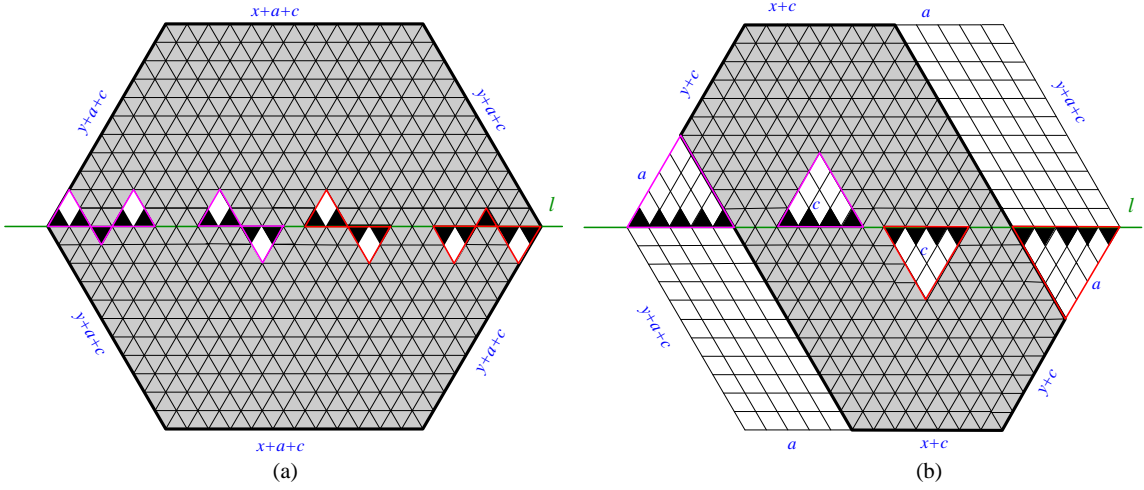}
\caption{Illustrating the proof of Theorem \ref{symmetricthreefern2}. Here $a=a_1+a_2+a_3+\cdots$ and $c=c_1+c_2+c_3+\cdots$.}\label{figsymfern3}
\end{figure}

\begin{thm}\label{symmetricthreefern2} For nonnegative integers $x,y$ and two ferns $F_1$ and $F_2$ consisting of triangles of side-lengths $a_1,a_2,\dots,a_{m}$ and $c_1,c_2,\dots,c_k$, from left to right, respectively. Assume in addition that $m$ is odd and $k$ is even (and the other cases can be obtained from the this case by appending a triangle of side-length $0$ to the end of the corresponding fern).

(a) If $x$ is odd and $y$ is even, then
\begin{align} \label{symfern2}
\M_c\left(E'_{x,y}\left(F_1, F_2\ |\  \frac{x+y-1}{2}\right)\right)&=\frac{s\left(a_1,\dots,a_{m},\frac{x+y-1}{2},c_1,\dots,c_{k}+1,c_{k},\dots,c_1+\frac{x+y-1}{2}+a_{m},a_{m-1},\dots,a_1\right)}{\PP\left(a,\frac{x+y-1}{2},c\right)}\notag\\&\times\M_c(B'_{y+2a+c,x+c,y+c,c})\notag\\
&=\frac{s\left(a_1,\dots,a_{m},\frac{x+y-1}{2},c_1,\dots,c_{k}+1,c_{k},\dots,c_1+\frac{x+y-1}{2}+a_{m},a_{m-1},\dots,a_1\right)}{\PP\left(a,\frac{x+y-1}{2},c\right)}\notag\\
&\times \PP \left ( \frac{x+2c+1}{2},\frac{y}{2}, c \right) \prod_{i=1}^{\frac{x-1}{2}}\frac{(\frac{y+2a}{2}+i)_{c}}{(i)_{c}}\notag\\
&\times \prod_{i=1}^{\frac{y}{2}}\frac{(\frac{y+2a+2c}{2}+i)_{\frac{x-1}{2}}}{(c+i)_{\frac{x-1}{2}}}\frac{(\frac{y+2a+2c}{2}+i)_{\frac{x+2c+1}{2}}}{(c+i)_{\frac{x+2c+1}{2}}}.
\end{align}
(b) If $x$ is even and $y$ is odd, then
\begin{align}\label{symfern3}
\M_c(E'_{x,y}(F_1, F_2\ |\  \frac{x+y-1}{2}))&=\frac{s(a_1,\dots,a_{m},\frac{x+y-1}{2},c_1,\dots,c_{k}+1,c_{k},\dots,c_1+\frac{x+y-1}{2}+a_{m},a_{m-1},\dots,a_1)}{\PP(a,\frac{x+y-1}{2},c)}\notag\\
&\times\M_c(B'_{y+2a+c,x+c,y+c,c})\notag\\
&=\frac{s(a_1,\dots,a_{m},\frac{x+y-1}{2},c_1,\dots,c_{k}+1,c_{k},\dots,c_1+\frac{x+y-1}{2}+a_{m},a_{m-1},\dots,a_1)}{\PP(a,\frac{x+y-1}{2},c)}\notag\\
 &\times \PP \left ( \frac{x+2c}{2},\frac{y-1}{2}, c +1 \right) \prod_{i=0}^{\frac{y-1}{2}}\frac{(\frac{y+2a+1}{2}+i)_{c}}{(c+i+1)_{c}}\notag\\
&\times \prod_{i=1}^{\frac{x}{2}}\frac{(\frac{y+2a-1}{2}+i)_{\frac{y+2c+1}{2}}}{(i)_{\frac{y+2c+1}{2}}}\frac{(\frac{y+2a+2c+1}{2}+c+i)_{\frac{y-1}{2}}}{(2c+i+1)_{\frac{y+2c+1}{2}}}.
\end{align}
\end{thm}

\begin{proof} We also pick  the fern $F_1$ consisting of $m$ triangles of side-length $a_1,a_2,\dots,a_m$ from left to right, starting by an up-pointing triangles, the fern $F_2$ consisting of $k$ triangles of side-lengths $c_1,c_2,\dots,c_k$ from left to right, starting by an up-pointing triangles, the ferns $F'_1$ and $F'_2$ consisting of a single up-pointing triangle of side-lengths $a$ and $c$, respectively.

We also apply Theorem \ref{symmetricfernthm2} to the region
\[E'_{x,y}\left(F_1,F_2\ |\  \frac{x+y-1}{2}\right)=R_{x,y}\left(F_1,F_2, \overline{F}_2, \overline{F_1}\ |\  \frac{x+y-1}{2},1,\frac{x+y-1}{2}\right)\]
and the region
 \[E'_{x,y}\left(F'_1,F'_2\ |\  \frac{x+y-1}{2}\right)=R_{x,y}\left(F'_1,F'_2, \overline{F'_2}, \overline{F'_1}\ | \  \frac{x+y-1}{2},1,\frac{x+y-1}{2}\right).\]
We note that, by removing forced lozenges from the region $E'_{x,y}(F'_1,F'_2\ |\ \frac{x+y-1}{2})$ as in Figure \ref{figsymfern3}(b), we get a region congruent with the region $B'_{y+2a+c,x+c,y+c,c}$ in  \cite[Theorem 5]{Ciu3}. This implies the first equalities in (\ref{symfern2}) and (\ref{symfern3}).  The second equality in each part of the above theorem follows from Ciucu's enumeration of centrally symmetric tilings of the region $B'_{x,y,z,k}$ given in  \cite[Theorem 5]{Ciu3}.
\end{proof}

\begin{rmk}We note that our region $E_{x,y}(F_1,F_2|\  \frac{x+y}{2})$ and $E'_{x,y}\left(F_1,F_2\ |\  \frac{x+y-1}{2}\right)$ can be considered as a common generalization of
\begin{enumerate}
\item the centrally symmetric fern-cored hexagon in \cite{Ciu4};
\item the region $B_{x,y,z,k}$ and $B'_{x,y,z,k}$ in \cite{Ciu3};
\item the centrally symmetric hexagon $Hex(x,y,z)$,
\end{enumerate}
up to removal of certain forced lozenges. Indeed, for the case of fern-cored hexagons in \cite{Ciu4}, we pick $F_1$ consisting of a single triangle (and $F_2$ being arbitrary). Then after removing forced lozenges similarly to that in Figures \ref{figsymfern2}(b) and \ref{figsymfern3}(b), we get back a centrally symmetric fern-cored hexagon in \cite{Ciu4}. Next, if assume in addition that $F_2$ consists of a single up-pointing triangle, then we get the region $B_{x,y,z,k}$ or $B_{x,y,z,k}$ as in Figures \ref{figsymfern2}(b) and \ref{figsymfern3}(b). Finally, if  we assume that $F_1$ consists of a single triangle and that $F_2$ is an empty fern, then our region becomes a centrally symmetric hexagon $Hex(x,y,z)$ after removing forced lozenges. This means that our Theorems \ref{symmetricfernthm} and \ref{symmetricfernthm2} can be considered as a common generalization of Ciucu's main results in \cite{Ciu4} and \cite{Ciu4} and Stanley's enumeration of self-complementary plane partitions (equivalently, centrally symmetric tilings of a hexagon) in \cite{Stanley}.
\end{rmk}

\section{Kuo Condensation and other preliminary results}
A \emph{forced lozenge} in a region $R$ is a lozenge that appears in any tilings of $R$. The removal of one or more forced lozenges does not change the number of tilings of the region. 


A \emph{perfect matching} (or simply \emph{matching} in this paper) of a graph is a collection of disjoint edges that covers all vertices of the graph.  A \emph{(planar) dual graph} of a region $R$ on the triangular lattice is the graph whose vertices are unit triangles in $R$ and whose edges connect precisely those two unit triangles of $R$ sharing an edge. The lozenge tilings of a region $R$ are in bijection with the matchings of its dual graph. Under this point of view, we use the notation $\M(G)$ (resp, $\M_c(G)$) for the number of matchings (resp., the number of centrally symmetric matchings) of the graph $G$. 

We will employ the following elegant variations of Kuo condensation introduced by Ciucu in \cite{Ciu3}\footnote{The author also obtained equivalent versions of Theorems \ref{ciukuo1} and \ref{ciukuo2} when working on the initial version of the paper. The proofs of the author are similar, \emph{but longer and more complicated} than that of Ciucu in \cite{Ciu3}.}. We refer the reader to Kuo's paper \cite{Kuo} for the original versions Kuo condensation.
\begin{thm}[Theorem 2 in \cite{Ciu3}]\label{ciukuo1} Let $G$ be a centrally symmetric, planar bipartite graph
embedded in an annulus. Let $\{a_1, a_2\}$, $\{b_1, b_2\}$ and $\{c_1, c_2\}$ be pairs of symmetric vertices on
the outer face of $G$. Assume that $a_1$, $b_1$, $c_1$, $a_2$, $b_2$, $c_2$ appear in this cyclic order around the
outer face, and that they alternate in color. Let $\{d_1, d_2\}$ be a pair of symmetric vertices of $G$
on the central face. Then we have
\begin{equation}\label{kuoeq}
\M_c(G)\M_c(G_{abcd}) = \M_c(G_{ab})\M_c(G_{cd}) +\M_c(G_{ac})\M_c(G_{bd}) +\M_c(G_{ad})\M_c(G_{bc}),
\end{equation}
where $G_{abcd} = G-\{a_1, a_2, b_1, b_2, c_1, c_2, d_1, d_2\}$, $G_{ab} = G-\{a_1, a_2, b_1, b_2\}$, and so on.
\end{thm}

We consider a variation of the above result, in which the vertices
$d_1$ and $d_2$ belong to two different faces, provided they share an edge. This follows directly from Theorem 3 in \cite{Ciu3}, when the two faces  $\mathcal{F}_1$ and $\mathcal{F}_2$ share
an edge.
\begin{thm}\label{ciukuo2} Let $G$ be a centrally symmetric, planar bipartite graph embedded in a
disk, so that $F_1$ and $F_2$ are two adjacent faces that are each other's image through the central
symmetry. Assume that $a_1$, $b_1$, $c_1$, $a_2$, $b_2$, $c_2$ are six vertices chosen as in Theorem \ref{ciukuo1} and that $d_1\in \mathcal{F}_1$ and $d_2\in \mathcal{F}_2$ are images of each other through the central symmetry. Then  we also have
\begin{equation}
\M_c(G)\M_c(G_{abcd}) = \M_c(G_{ab})\M_c(G_{cd}) +\M_c(G_{ac})\M_c(G_{bd}) +\M_c(G_{ad})\M_c(G_{bc}).
\end{equation}
\end{thm}

Strictly speaking, as originally stated in \cite{Ciu3}, Theorems \ref{ciukuo1} and \ref{ciukuo2} still holds for the case when $G$ is a weighted graph, i.e. the edges of $G$ carry weights. However, we only consider unweighted graphs in this paper.

\section{Proof of the main theorem (Theorem \ref{centralfactor})}

We would like to show that
\begin{align}\label{maineqrefine}
\M_c(CS_{x,y}(U;D;B))=\frac{\Delta(U\cup \overline{D})}{\Delta(U'\cup \overline{D'})}\cdot \M_c(CS_{x,y}(U';D';B)).
\end{align}
here we  use the shorthand notation $\overline{D}$ for $((x+y+2n+1)-D)$.

 It would be convenient to view the numerator of the fraction on the right hand-side of (\ref{maineqrefine}) is the produce of differences of elements in the index set of removed up-pointing triangles in the region $CS_{x,y}(U;D;B)$, and the denominator is the product of differences of the elements in the index set of removed up-pointing triangles in the region $CS_{x,y}(U';D';B)$.

There are two cases to distinguish, the case when $x,y$ have the same parity and the case when they have opposite parities. We first consider the case when $x,y$ have the same parity. It is not hard to see that the region $CS_{x,y}(U;D;B)$ admits a centrally symmetric tiling only if $x,y$ are both even. Therefore, we assume, without loss of generality, that $x,y$ are both even.

\medskip

We prove (\ref{maineqrefine}) by the induction on $x+y+u+d$. The base cases are the situations when $x<2b+2$, when $y<2$, and when $u+d=0$.

\begin{figure}\centering
\includegraphics[width=10cm]{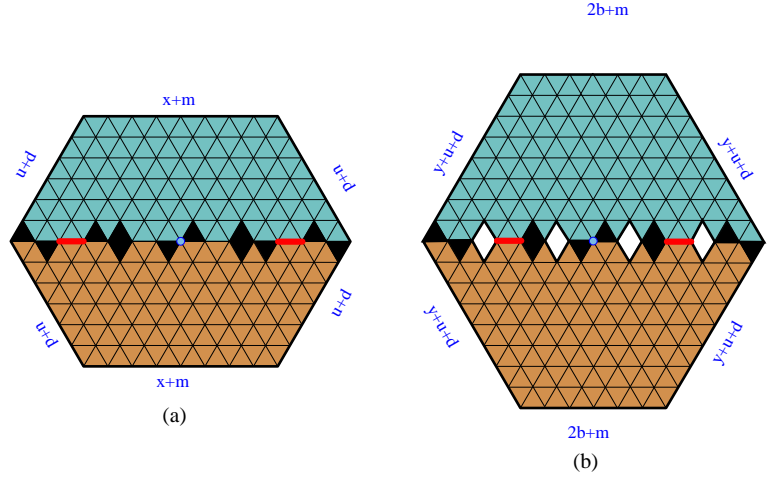}
\caption{The base cases when $x+y$ is even: (a) $y=0$ and (b) $x=2b$.}\label{basecase1}
\end{figure}
As the case $u+d=0$ is trivial, we consider now the case when $y<2$. Since $y$ is even, $y$ must be $0$. Each centrally symmetric tiling of our region can be partitioned into tilings of two congruent dented semihexagons obtained by dividing our region along the horizontal axis. These dented semihexagons that are images of each other through the central symmetry (see Figure \ref{basecase1}(a)). It easy to see that the centrally symmetric tilings of the region are in bijection with tilings of each of the dented semihexagons. This way we have
 \[\M_c(CS_{x,0}(U;D;B))=\M(T_{x+m,u+d}(U \cup \overline{D})).\] Similarly, we have
 \[\M_c(CS_{x,0}(U';D';B))=\M(T_{x+m,u+d}(U' \cup\overline{D'})).\]
  Then (\ref{maineqrefine})  follows directly from Cohn--Larsen--Propp's formula (\ref{CLPeq}).

Next, we consider the case $x<2b+2$. As $x$ is even and $x\geq 2b$, $x$ must be $2b$. In each centrally symmetric tilings of our region, there is a vertical lozenge at each of the positions in the complement of $\mathcal{O}:=(U \cup D\cup B)\cup\overline{(U \cup D\cup B)}$. This way the symmetric tiling can be partitioned into tilings of two congruent dented semihexagons that are images of each other through the central symmetry (see Figure \ref{basecase1}(b)). This means that
\[\M_c(CS_{2b,y}(U;D;B))=\M(T_{2b+m,y+u+d}(\mathcal{O}^c\cup U \cup \overline{D})).\]
Similarly, we get
\[\M_c(CS_{2b,y}(U';D';B))=\M(T_{2b+m,y+u+d}(\mathcal{O}^c\cup U' \cup \overline{D'})).\]
Then(\ref{maineqrefine})  follows again from Cohn--Larsen--Propp's formula (\ref{CLPeq}), after performing a straight forward simplification.

For the induction step, we assume that $u+d>0$, $x\geq 2b+2$, $y\geq 2$, and that the identity (\ref{maineqrefine}) holds for any  $CS$-type regions whose sum of $x$-, $y$-, $u$-, and $d$-parameters is strictly less than $x+y+u+d$. We first use Kuo condensation in Theorem \ref{ciukuo1} to set up a recurrence for  the left-hand side of (\ref{maineqrefine}), and we show that the expression on the right-hand side satisfies the same recurrence. Then (\ref{maineqrefine}) follows from the induction principle.

\medskip

In the rest of the proof we use the shorthand notation $\mathcal{O}$ for the position set of all `obstacles' (i.e., removed unit triangles and barriers) of the region $CS_{x,y}(U;D;B)$, i.e. $\mathcal{O}:=(U \cup D\cup B)\cup\overline{(U \cup D\cup B)}$.

\medskip

There are four subcases to distinguish here:

\medskip

\begin{figure}\centering
\includegraphics[width=11cm]{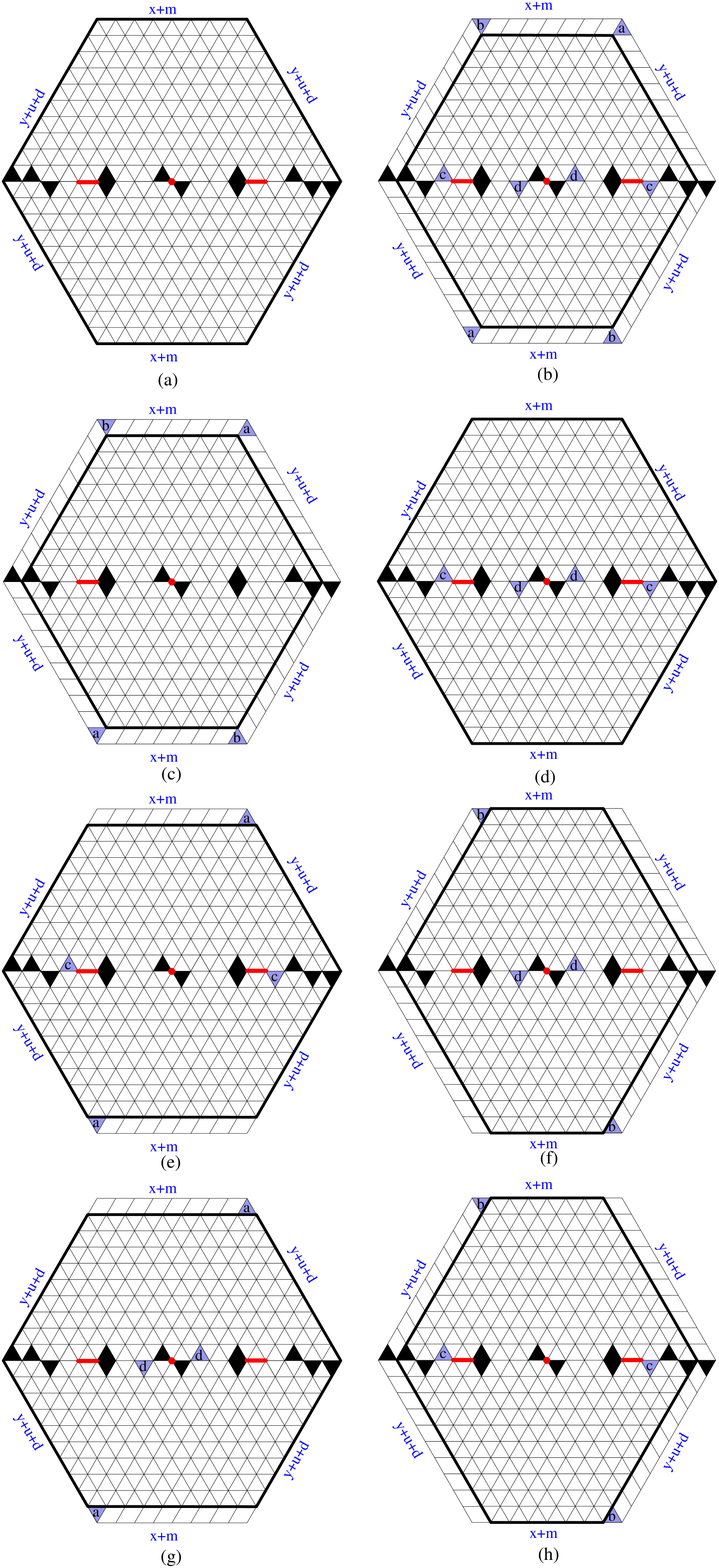}
\caption{How to apply Kuo condensation in the case when $1\in U$ and $1\notin D$.}\label{SymmetricFern1}
\end{figure}

\textbf{Case 1. $1\in U\setminus D$.}
\
\\

We apply Kuo condensation in Theorem \ref{ciukuo1} to the dual graph $G$ of the region $CS_{x,y}(U;D;B)$ with the choice of the eight vertices $a_1,a_2,b_1,b_2,c_1,c_2,d_1,d_2$ shown in Figure \ref{SymmetricFern1}(b). We assume, by convention, that  the dual graph of a centrally symmetric region is also centrally symmetric, in particular $G$ is centrally symmetric. More precisely, the figure shows the positions of the unit triangles in the region $CS_{x,y}(U;D;B)$ that correspond to the vertices $a_1,a_2,b_1,b_2,c_1,c_2,d_1,d_2$ of $G$.  The triangles corresponding to $a_1, a_2$ are both labelled by $a$, the ones corresponding to $b_1,b_2$ are both labelled by $b$, and so on. In particular, the unit triangle corresponding to the vertices $a_1,a_2$ are the up-pointing shaded triangle on the northeast corner and the down-pointing shaded one on the southwest corner of the region. The $b_1$- and $b_2$-triangles are the shaded unit triangles on the northwest and southeast corners. We pick the $c_1$- and $c_2$-triangles along the horizontal axis $l$ at the first and the last positions in  $\mathcal{O}^c$, and $d_1$- and $d_2$-triangles at the last position before the symmetric center and the first position after the symmetric center that are in  $\mathcal{O}^c$. Assume that the first position, that is not in $\mathcal{O}$, is $\alpha$, and the last position before the symmetric center, that is not in $\mathcal{O}$, is $\beta$. We have in particular $1<\alpha\leq  \beta$.

We note that, as we assuming that $2b+1<x$, the $c_1,c_2,d_1,d_2$-triangles are well-defined. We also note that $\alpha$ and $\beta$ may be equal, and this does not violate the structure of our doubly-dented hexagons.

\medskip

For a set $S$, we use the shorthand notations $u S$, $v S$, and $uvS$ for the unions $S\cup \{u\}$, $S\cup \{v\}$ and $S\cup \{u,v\}$, respectively.

Let us consider the region corresponding to the graph $G_{abcd}$, i.e., the region obtained from $CS_{x,y}(U;D;B)$ by removing the eight unit triangles corresponding to the eight vertices $a_1,a_2,b_1,b_2, c_1,c_2,d_1,d_2$. The removal of $a_1$-, $a_2$-, $b_1$-, $b_2$-triangles yields forced lozenges along the boundary of the region, while the removal of $c_1$-, $c_2$-, $d_1$-, $d_2$-triangles creates new `dents' along the horizontal axis $l$ of our doubly-dented hexagon. After removing forced lozenges, the leftover region is a new doubly-dented hexagon. The new $U$-index set is $(\alpha U\setminus\{1\})^\leftarrow$, where we use the notation $S^{\leftarrow}$ for the index set obtained from shifting all elements in $S$ a unit to the left (provided that $1\notin S$). Similarly, the $D$-index set of the new region is now $\beta D^{\leftarrow}$, and the $B$-index set is now $B^{\leftarrow}$. More precisely, our new region is $CS_{x-2,y-2}((\alpha U\setminus \{1\})^\leftarrow; \beta D^{\leftarrow};B^{\leftarrow})$. As the removal of forced lozenges does not change the tiling number,  we get
\begin{equation}\label{eq1b}
\M_c(G_{abcd})=\M_c(CS_{x-2,y-2}((\alpha U\setminus \{1\})^\leftarrow; \beta D^{\leftarrow}; B^{\leftarrow})).
\end{equation}
Working similarly on the regions corresponding with the other terms in the recurrence (\ref{kuoeq}) as shown in Figures \ref{SymmetricFern1}(c)--(h), we have
\begin{equation}\label{eq1c}
\M_c(G_{ab})=\M_c(CS_{x,y}((U\setminus \{1\})^\leftarrow; D^{\leftarrow}; B^{\leftarrow})),
\end{equation}
\begin{equation}\label{eq1d}
\M_c(G_{cd})=\M_c(CS_{x-2,y-2}(\alpha U; \beta D: B)),
\end{equation}
\begin{equation}\label{eq1e}
\M_c(G_{ac})=\M_c(CS_{x,y-2}(\alpha U; D; B)),
\end{equation}
\begin{equation}\label{eq1f}
\M_c(G_{bd})=\M_c(CS_{x-2,y}((U\setminus \{1\})^\leftarrow; \beta D^{\leftarrow}; B^{\leftarrow})),
\end{equation}
\begin{equation}\label{eq1g}
\M_c(G_{ad})=\M_c(CS_{x,y-2}(U; \beta D; B)),
\end{equation}
\begin{equation}\label{eq1h}
\M_c(G_{bc})=\M_c(CS_{x-2,y}((\alpha U\setminus \{1\})^\leftarrow; D^{\leftarrow};B^{\leftarrow})).
\end{equation}
Plugging the above equations to the recurrence in Theorem \ref{ciukuo1}, we get the following recurrence for the left-hand side of  (\ref{maineqrefine}):
\begin{align}\label{recurrence1}
\M_c(CS_{x,y}(U,D;B))&\M_c(CS_{x-2,y-2}((\alpha U\setminus \{1\})^\leftarrow; \beta D^{\leftarrow};B^{\leftarrow}))\notag\\
&=\M_c(CS_{x,y}((U\setminus \{1\})^\leftarrow; D^{\leftarrow}; B^{\leftarrow}))\M_c(CS_{x-2,y-2}(\alpha U; \beta D; B))\notag\\
&+\M_c(CS_{x,y-2}(\alpha U; D; B))\M_c(CS_{x-2,y}((U\setminus \{1\})^\leftarrow; \beta D^{\leftarrow}; B^{\leftarrow}))\notag\\
&+\M_c(CS_{x,y-2}(U; \beta D; B))\M_c(CS_{x-2,y}((\alpha U\setminus \{1\})^\leftarrow; D^{\leftarrow}; B^{\leftarrow})).
\end{align}

Next, we show that the expression on the right-hand side of (\ref{maineqrefine}) also satisfies  recurrence (\ref{recurrence1}). Equivalently, we need to verify
\begin{align}\label{recurrence1c}
&\frac{\frac{\Delta(((U\setminus \{1\})\cup \overline{D})^{\leftarrow})}{\Delta(((U'\setminus \{1\})\cup \overline{D'})^{\leftarrow})}\M_c(CS_{x,y}((U'\setminus \{1\})^\leftarrow; D'^{\leftarrow}; B^{\leftarrow}))\frac{\Delta(\alpha U\cup \overline{\beta D})}{\Delta(\alpha U'\cup \overline{\beta D'})}\M_c(CS_{x-2,y-2}(\alpha U'; \beta D'; B))}{\frac{\Delta(U\cup \overline{D})}{\Delta(U'\cup \overline{D'})}\M_c(CS_{x,y}(U';D';B))\frac{\Delta(((\alpha U\setminus \{1\})\cup \overline{\beta D})^{\leftarrow})}{\Delta(((\alpha U'\setminus \{1\})\cup \overline{\beta D'})^{\leftarrow})}\M_c(CS_{x-2,y-2}((\alpha U'\setminus \{1\})^\leftarrow; \beta D'^{\leftarrow};B^{\leftarrow}))}+\notag\\
&+\frac{\frac{\Delta(\alpha U\cup \overline{D})}{\Delta(\alpha U'\cup \overline{D}')}\M_c(CS_{x,y-2}(\alpha U'; D'; B))\frac{\Delta(((U\setminus \{1\})\cup \overline{\beta D})^{\leftarrow})}{\Delta(((U'\setminus \{1\})\cup \overline{\beta D'})^{\leftarrow})}\M_c(CS_{x-2,y}((U'\setminus \{1\})^{\leftarrow}; \beta D'^{\leftarrow}; B^{\leftarrow}))}{\frac{\Delta(U\cup \overline{D})}{\Delta(U'\cup \overline{D'})}\M_c(CS_{x,y}(U';D';B))\frac{\Delta(((\alpha U\setminus \{1\})\cup \overline{\beta D})^{\leftarrow})}{\Delta(((\alpha U'\setminus \{1\})\cup \overline{\beta D'})^{\leftarrow})}\M_c(CS_{x-2,y-2}((\alpha U'\setminus \{1\})^\leftarrow; \beta D'^{\leftarrow};B^{\leftarrow}))}\notag\\
&+\frac{\frac{\Delta(U\cup \overline{\beta D})}{\Delta(U'\cup \overline{\beta D'})}\M_c(CS_{x,y-2}(U'; \beta D'; B))\frac{\Delta(((\alpha U\setminus \{1\})\cup \overline{D})^{\leftarrow})}{\Delta(((\alpha U'\setminus \{1\})\cup \overline{D'})^{\leftarrow})}\M_c(CS_{x-2,y}((\alpha U'\setminus \{1\})^\leftarrow; D'^{\leftarrow}; B^{\leftarrow}))}{\frac{\Delta(U\cup \overline{D})}{\Delta(U'\cup \overline{D'})}\M_c(CS_{x,y}(U';D';B))\frac{\Delta(((\alpha U\setminus \{1\})\cup \overline{\beta D})^{\leftarrow})}{\Delta(((\alpha U'\setminus \{1\})\cup \overline{\beta D'})^{\leftarrow})}\M_c(CS_{x-2,y-2}((\alpha U'\setminus \{1\})^\leftarrow; \beta D'^{\leftarrow};B^{\leftarrow}))}=1.
\end{align}
We note that, as shown in Figures \ref{SymmetricFern1}(b), (c), (f), and (h), the index sets of removed up-pointing triangles in the regions $CS_{x-2,y-2}((\alpha U\setminus \{1\})^\leftarrow; \beta D^{\leftarrow};B^{\leftarrow})$, $CS_{x,y}((U\setminus \{1\})^\leftarrow; D^{\leftarrow}; B^{\leftarrow})$, $CS_{x-2,y}((U\setminus \{1\})^\leftarrow; \beta D^{\leftarrow}; B^{\leftarrow})$, $CS_{x-2,y}((\alpha U\setminus \{1\})^\leftarrow; D^{\leftarrow}; B^{\leftarrow})$ are respectively $((\alpha U\setminus \{1\})\cup \overline{\beta D})^{\leftarrow}$, $((U\setminus \{1\})\cup \overline{D})^{\leftarrow}$, $((\alpha U'\setminus \{1\})\cup \overline{\beta D'})^{\leftarrow}$, and $((\alpha U\setminus \{1\})\cup \overline{D})^{\leftarrow}$.

\medskip

We claim that
\begin{claim}\label{clm1}
\begin{equation}\label{ratio1}
\dfrac{\Delta(((U\setminus \{1\})\cup \overline{D})^{\leftarrow})}{\Delta(((U'\setminus \{1\})\cup \overline{D'})^{\leftarrow})}\dfrac{\Delta(\alpha U\cup \overline{\beta D})}{\Delta(\alpha U'\cup \overline{\beta D'})}=\dfrac{\Delta(U\cup \overline{D})}{\Delta(U'\cup \overline{D'})}\dfrac{\Delta(((\alpha U\setminus \{1\})\cup \overline{\beta D})^{\leftarrow})}{\Delta(((\alpha U'\setminus \{1\})\cup \overline{\beta D'})^{\leftarrow})},
\end{equation}

\begin{equation}\label{ratio2}
\dfrac{\Delta(\alpha U\cup \overline{D})}{\Delta(\alpha U'\cup \overline{D}')}\dfrac{\Delta(((U\setminus \{1\})\cup \overline{\beta D})^{\leftarrow})}{\Delta(((U'\setminus \{1\})\cup \overline{\beta D'})^{\leftarrow})}=\dfrac{\Delta(U\cup \overline{D})}{\Delta(U'\cup \overline{D'})}\dfrac{\Delta(((\alpha U\setminus \{1\})\cup \overline{\beta D})^{\leftarrow})}{\Delta(((\alpha U'\setminus \{1\})\cup \overline{\beta D'})^{\leftarrow})},
\end{equation}
 and
\begin{equation}\label{ratio3}
\dfrac{\Delta(U\cup \overline{\beta D})}{\Delta(U'\cup \overline{\beta D'})}\dfrac{\Delta(((\alpha U\setminus \{1\})\cup \overline{D})^{\leftarrow})}{\Delta(((\alpha U'\setminus \{1\})\cup \overline{D'})^{\leftarrow})}=\dfrac{\Delta(U\cup \overline{D})}{\Delta(U'\cup \overline{D'})}\dfrac{\Delta(((\alpha U\setminus \{1\})\cup \overline{\beta D})^{\leftarrow})}{\Delta(((\alpha U'\setminus \{1\})\cup \overline{\beta D'})^{\leftarrow})}.
\end{equation}
\end{claim}
\begin{proof}[Proof of Claim \ref{clm1}]
Let us verify (\ref{ratio1}). As the index shifting does not have any effects on the operation $\Delta$, we can rewrite (\ref{ratio1}) as
\begin{equation}
\frac{\Delta(U\setminus \{1\}\cup \overline{D})\Delta(\alpha U\cup \overline{\beta D})}{\Delta(U\cup \overline{D})\Delta(\alpha U\setminus \{1\}\cup \overline{\beta D})}=\frac{\Delta(U'\setminus \{1\}\cup \overline{D'})\Delta(\alpha U'\cup \overline{\beta D'})}{\Delta(U'\cup \overline{D'})\Delta(\alpha U'\setminus \{1\}\cup \overline{\beta D'})}.
\end{equation}
Let us simplify the fraction on the left-hand side:
\begin{align}
&\frac{\Delta(U\setminus \{1\}\cup \overline{D})\Delta(\alpha U\cup \overline{\beta D})}{\Delta(U\cup \overline{D})\Delta(\alpha U\setminus \{1\}\cup \overline{\beta D})}\notag\\
&=\frac{\Delta(U\setminus \{1\}\cup \overline{D})}{\Delta(U\cup \overline{D})}\frac{\Delta(\alpha U\cup \overline{\beta D})}{\Delta(\alpha U\setminus \{1\}\cup \overline{\beta D})}\notag\\
&=\frac{(\alpha-1)(\overline{\beta}-1)\prod_i(s_i-1)\prod_j(t_j-1)}{\prod_i(s_i-1)\prod_j(\overline{t_j}-1)}\notag\\
&=(\alpha-1)(\overline{\beta}-1).
\end{align}
The second equality holds by cancelling out the common terms in the numerator and denominator of each fraction after the first equality sign.
Here, for any index $i$ of $[x+y+2n]$,  we use the notation $\overline{i}$ for the image of $i$ through the central symmetry, i.e. $\overline{i}:=(x+y+2n+1)-i$. Similarly, one can simplify the  faction on the right-hand side to $(\alpha-1)(\overline{\beta}-1)$, and (\ref{ratio1}) follows.

Let us consider (\ref{ratio2}). As the index shifting has no effect on the operation $\Delta$, we can write (\ref{ratio2}) as
\begin{equation}
\frac{\Delta(\alpha U\cup \overline{D})\Delta(U\setminus \{1\}\cup \overline{\beta D})}{\Delta(U\cup \overline{D})\Delta(\alpha U\setminus \{1\}\cup \overline{\beta D})}= \frac{\Delta(\alpha U'\cup \overline{D}')\Delta(U'\setminus \{1\}\cup \overline{\beta D'})}{\Delta(U'\cup \overline{D'})\Delta(\alpha U'\setminus \{1\}\cup \overline{\beta D'})}.
\end{equation}
The fraction on the left-hand side can be simplified as
\begin{align}
&\frac{\Delta(\alpha U\cup \overline{D})\Delta(U\setminus \{1\}\cup \overline{\beta D})}{\Delta(U\cup \overline{D})\Delta(\alpha U\setminus \{1\}\cup \overline{\beta D})}\notag\\
&=\frac{\prod_{i}|s_i-\alpha|\prod_{j}|\overline{t_j}-\alpha| }{|\overline{\beta}-\alpha|\prod_{i>1}|s_i-\alpha|\prod_{j}|\overline{t_j}-\alpha|}\notag\\
&=\frac{|\alpha-1|}{|\overline{\beta}-\alpha|}.
\end{align}
Similarly, the right-hand side is equal to $\frac{|\alpha-1|}{|\overline{\beta}-\alpha|}$, and  (\ref{ratio2}) follows.

Finally, let us work on (\ref{ratio3}). Eliminating the shifting sign and rearranging, we have (\ref{ratio3}) is equivalent to
\begin{equation}\label{ratio3b}
\frac{\Delta(U\cup \overline{\beta D})\Delta(\alpha U\setminus \{1\}\cup \overline{D})}{\Delta(U\cup \overline{D})\Delta(\alpha U\setminus \{1\}\cup \overline{\beta D})}=\frac{\Delta(U'\cup \overline{\beta D'})\Delta(\alpha U'\setminus \{1\}\cup \overline{D'})}{\Delta(U'\cup \overline{D'})\Delta(\alpha U'\setminus \{1\}\cup \overline{\beta D'})}.
\end{equation}
The left-hand side can be simplified as
\begin{align}
&\frac{\Delta(U\cup \overline{\beta D})\Delta(\alpha U\setminus \{1\}\cup \overline{D})}{\Delta(U\cup \overline{D})\Delta(\alpha U\setminus \{1\}\cup \overline{\beta D})}\notag\\
&=\frac{\prod_{i}|s_i-\overline{\beta}|\prod_{j}|t_j-\overline{\beta}|}{|\alpha-\overline{\beta}|\prod_{i>1}|s_i-\overline{\beta}|\prod_{j}|t_j-\overline{\beta}|}\notag\\
&=\frac{|1-\overline{\beta}|}{|\alpha-\overline{\beta}|}.
\end{align}
Similarly, the right-hand side is also equal to $\frac{|1-\overline{\beta}|}{|\alpha-\overline{\beta}|}$, and (\ref{ratio3}) follows.
\end{proof}

By (\ref{ratio1})--(\ref{ratio3}),  we have (\ref{recurrence1c}) reduced to
\begin{align}\label{recurrence1d}
&\frac{\M_c(CS_{x,y}((U'\setminus \{1\})^\leftarrow; D'^{\leftarrow}; B^{\leftarrow}))\M_c(CS_{x-2,y-2}(\alpha U'; \beta D'; B))}{\M_c(CS_{x,y}(U',D';B))\M_c(CS_{x-2,y-2}((\alpha U'\setminus \{1\})^\leftarrow; \beta D'^{\leftarrow};B^{\leftarrow}))}+\notag\\
&+\frac{\M_c(CS_{x,y-2}(\alpha U'; D'; B))\M_c(CS_{x-2,y}((U\setminus \{1\})^\leftarrow; \beta D'^{\leftarrow}; B^{\leftarrow}))}{\M_c(CS_{x,y}(U',D';B))\M_c(CS_{x-2,y-2}((\alpha U'\setminus \{1\})^\leftarrow; \beta D'^{\leftarrow};B^{\leftarrow}))}\notag\\
&+\frac{\M_c(CS_{x,y-2}(U'; \beta D'; B))\M_c(CS_{x-2,y}((\alpha U'\setminus \{1\})^\leftarrow; D'^{\leftarrow}; B^{\leftarrow}))}{\M_c(CS_{x,y}(U',D';B))\M_c(CS_{x-2,y-2}((\alpha U'\setminus \{1\})^\leftarrow; \beta D'^{\leftarrow};B^{\leftarrow}))}=1,
\end{align}
equivalently,
\begin{align}\label{recurrence1e}
\M_c(CS_{x,y}(U';D';B))&\M_c(CS_{x-2,y-2}((\alpha U'\setminus \{1\})^\leftarrow; \beta D'^{\leftarrow};B^{\leftarrow}))\notag\\
&=\M_c(CS_{x,y}((U'\setminus \{1\})^\leftarrow; D'^{\leftarrow}; B^{\leftarrow}))\M_c(CS_{x-2,y-2}(\alpha U'; \beta D'; B))\notag\\
&+\M_c(CS_{x,y-2}(\alpha U'; D'; B))\M_c(CS_{x-2,y}((U'\setminus \{1\})^\leftarrow; \beta D'^{\leftarrow}; B^{\leftarrow}))\notag\\
&+\M_c(CS_{x,y-2}(U'; \beta D'; B))\M_c(CS_{x-2,y}((\alpha U'\setminus \{1\})^\leftarrow; D'^{\leftarrow}; B^{\leftarrow})),
\end{align}
which follows directly from the application of recurrence (\ref{recurrence1}) to the region $CS_{x,y}(U';D';B)$. This finishes our verification that the right-hand side of (\ref{maineqrefine}) satisfies  recurrence (\ref{recurrence1}) and finishes our proof in the case when $1\in U\setminus D$.

\bigskip

\textbf{Case 2. $1\in D\setminus U$.} This case can be reduced to Case 1 treated above by rotating the region $180$ degree.

\bigskip

\textbf{Case 3. $1\notin U \cup D$.}

If $1 \in B$, then there are forced lozenges along the northwest and southwest side of the region. After removing these forced lozenges, our region becomes $CS_{x-2,y}(U^{\leftarrow}; D^{\leftarrow}; (B\setminus \{1\})^{\leftarrow})$, and (\ref{maineqrefine}) follows from the induction hypothesis. Without loss of generality, we assume in the rest of Case 3 that $1\notin B$.

\bigskip

\begin{figure}\centering
\includegraphics[width=11cm]{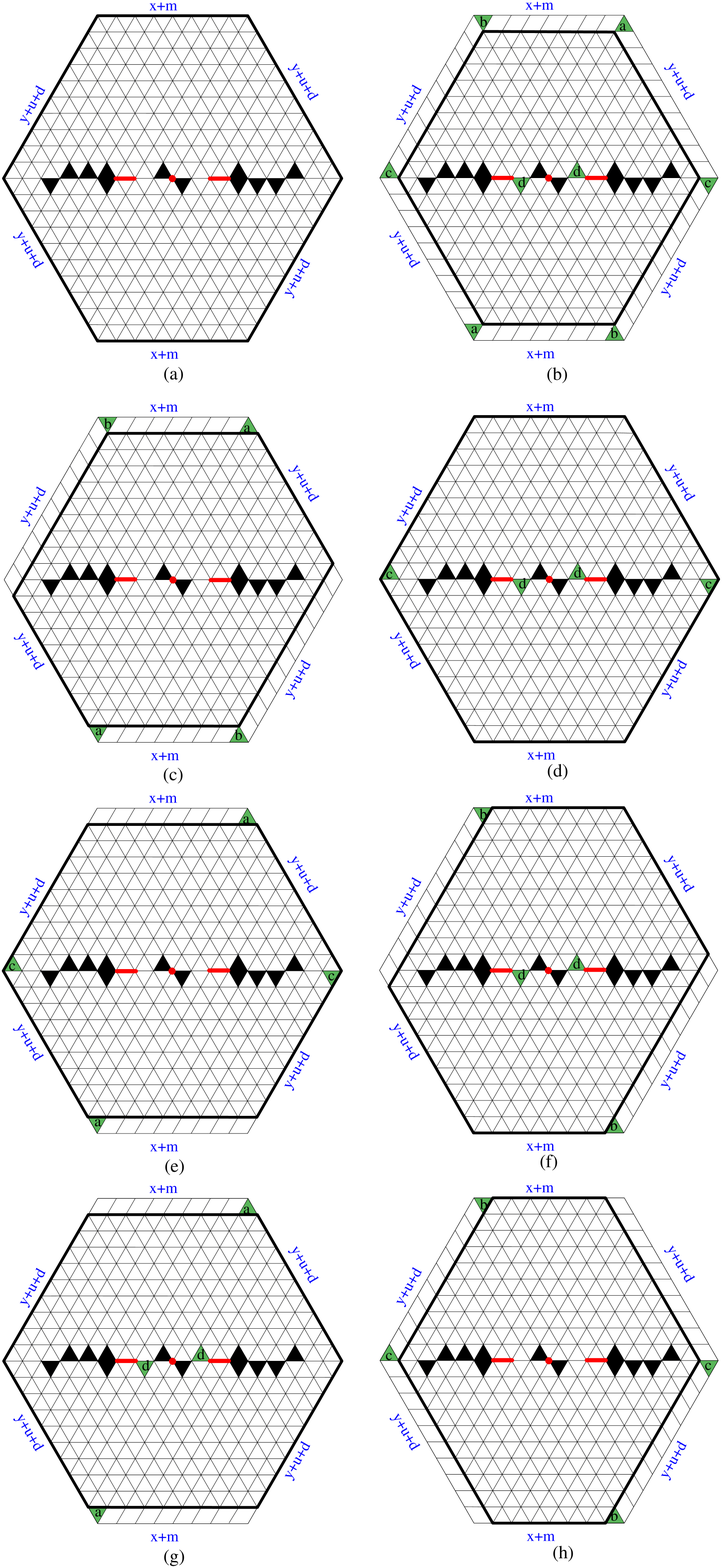}
\caption{How to apply Kuo condensation in the case when $1\notin U\cup D$.}\label{SymmetricFern2}
\end{figure}

\begin{figure}\centering
\includegraphics[width=11cm]{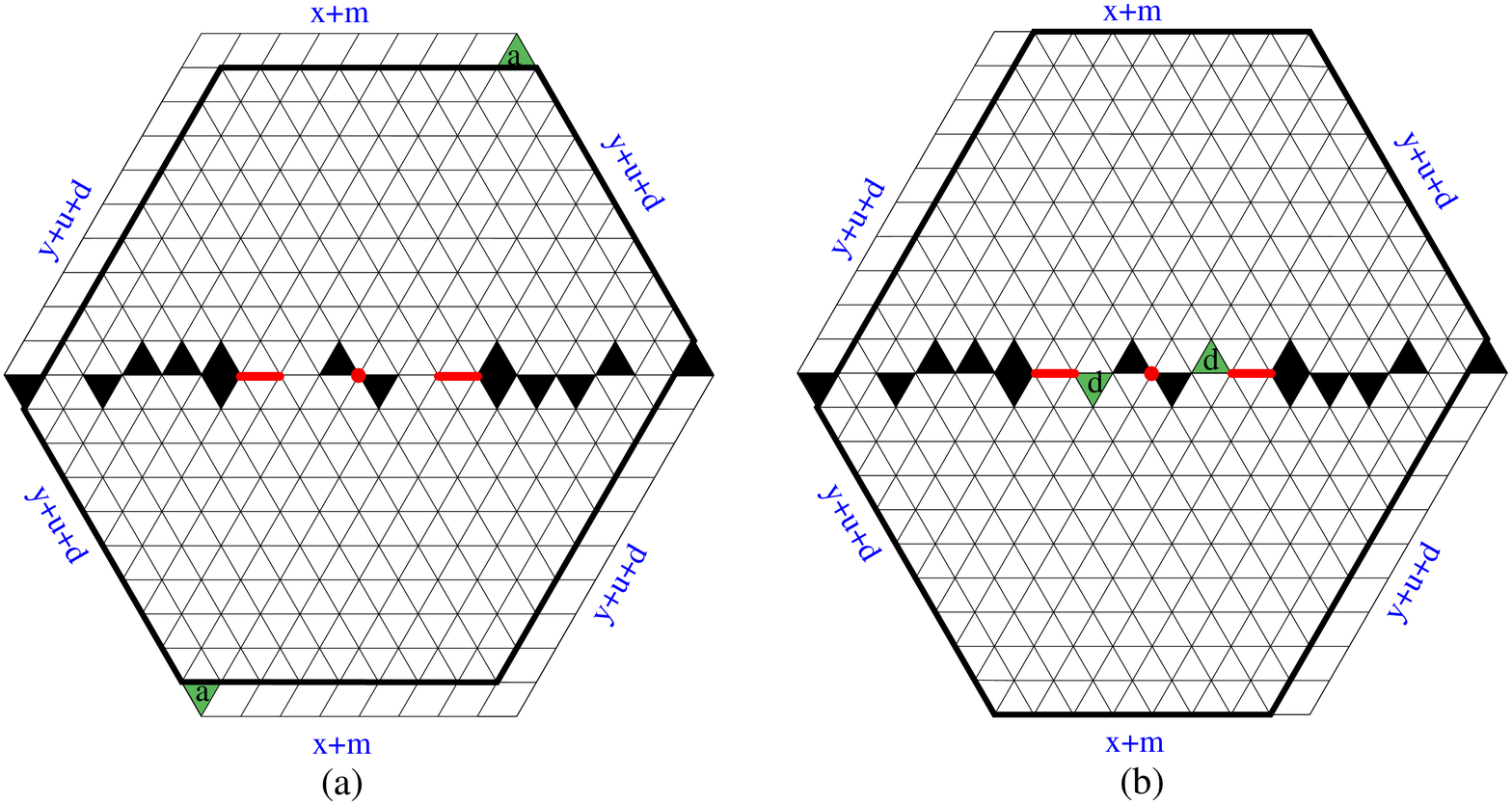}
\caption{Reforming the leftover regions in the case of the regions corresponding to (a) $G_{ab}$ and (b) $G_{bd}$.}\label{SymmetricFern2b}
\end{figure}

Similar to Case 1, we apply Kuo condensation in Theorem \ref{ciukuo1} with the choice of the eight vertices $a_1,a_2,b_1,b_2,c_1,c_2,d_1,d_2$ as shown in Figure \ref{SymmetricFern2}(b). We note that the choice of these vertices are the same as in Case 1, as we now have $\alpha=1$.

Consider the region corresponding with the graph $G_{abcd}$. The removal of the unit triangles corresponding to $a_1,a_2,b_1,b_2,c_1,c_2$ yields forced lozenges around the boundary of the region (see Figure \ref{SymmetricFern2}(b)). After removing the forced lozenges from the region, we get the region $CS_{x-2,y-2}(U^{\leftarrow};\beta D^{\leftarrow}; B^{\leftarrow})$. This means that
\begin{equation}\label{eq2b}
\M_c(G_{abcd})=\M_c(CS_{x-2,y-2}(U^{\leftarrow};\beta D^{\leftarrow}; B^{\leftarrow})).
\end{equation}
Working similarly on the regions corresponding with the others terms in the recurrence (\ref{kuoeq}) as shown in Figures \ref{SymmetricFern1}(c)--(h), we have
\begin{equation}\label{eq2c}
\M_c(G_{ab})=\M_c(CS_{x,y-2}(U; 1D; B))
\end{equation}
(note that $\alpha=1$ in this case),
\begin{equation}\label{eq2d}
\M_c(G_{cd})=\M_c(CS_{x-2,y-2}(1U; \beta D;B)),
\end{equation}
\begin{equation}\label{eq2e}
\M_c(G_{ac})=\M_c(CS_{x,y-2}(1U; D;B)),
\end{equation}
\begin{equation}\label{eq2f}
\M_c(G_{bd})=\M_c(CS_{x-2,y-2}(U; 1\beta D;B)),
\end{equation}
\begin{equation}\label{eq2g}
\M_c(G_{ad})=\M_c(CS_{x,y-2}(U; \beta D;B)),
\end{equation}
\begin{equation}\label{eq2h}
\M_c(G_{bc})=\M_c(CS_{x-2,y}(U^{\leftarrow};D^{\leftarrow};B^{\leftarrow})).
\end{equation}
Strictly speaking, the leftover regions after removing forced lozenges in the cases of $G_{ab}$ and $G_{bd}$, as shown in Figures \ref{SymmetricFern2}(c) and (f),  are \emph{not} doubly--dented regions. However, they have the same tiling numbers as the regions
$CS_{x,y-2}(U; 1D)$ and $CS_{x-2,y-2}(U; 1\beta D)$, respectively. Indeed, as shown in Figures \ref{SymmetricFern2b}(a) and (b), the latter two regions, after removing forced lozenges, are congruent with the leftover regions in  Figures \ref{SymmetricFern2}(c) and (f). This yields idenitities (\ref{eq2c}) and (\ref{eq2f}) above.

Therefore, by plugging the above  identities to the recurrence in Theorem \ref{ciukuo1}, we have the following recurrence:
\begin{align}\label{recurrence2}
&\M_c(CS_{x,y}(U,D;B))\M_c(CS_{x-2,y-2}(U^{\leftarrow};\beta D^{\leftarrow};B^{\leftarrow}))=\M_c(CS_{x,y-2}(U; 1D;B))\M_c(CS_{x-2,y-2}(1U; \beta D;B))\notag\\
&+\M_c(CS_{x,y-2}(1U; D;B))\M_c(CS_{x-2,y-2}(U; 1\beta D;B))+\M_c(CS_{x,y-2}(U; \beta D;B))\M_c(CS_{x-2,y}(U^{\leftarrow}; D^{\leftarrow};B^{\leftarrow})).
\end{align}

Next, we show that the right-hand side of (\ref{maineqrefine}) also satisfies recurrence (\ref{recurrence2}) above. Equivalently, we need to verify that:
\begin{align}\label{recurrence2c}
&\frac{\frac{\Delta(U\cup \overline{1D})}{\Delta(U'\cup \overline{1D'})}\M_c(CS_{x,y-2}(U'; 1D';B))\frac{\Delta(1U\cup \overline{\beta D})}{\Delta(1U'\cup \overline{\beta D'})}\M_c(CS_{x-2,y-2}(1U'; \beta D';B))}{\frac{\Delta(U\cup \overline{D})}{\Delta(U'\cup \overline{D'})}\M_c(CS_{x,y}(U',D';B))\frac{\Delta((U\cup \overline{\beta D})^{\leftarrow})}{\Delta((U'\cup \overline{\beta D'})^{\leftarrow})}\M_c(CS_{x-2,y-2}(U'^{\leftarrow};\beta D'^{\leftarrow};B^{\leftarrow}))}\notag\\
&+\frac{\frac{\Delta(1U\cup \overline{D})}{\Delta(1U'\cup \overline{D'})}\M_c(CS_{x,y-2}(1U'; D';B))\frac{\Delta(U\cup \overline{1\beta D})}{\Delta(U'\cup \overline{1\beta D'})}\M_c(CS_{x-2,y-2}(U'; 1\beta D';B))}{\frac{\Delta(U\cup \overline{D})}{\Delta(U'\cup \overline{D'})}\M_c(CS_{x,y}(U',D';B))\frac{\Delta((U\cup \overline{\beta D})^{\leftarrow})}{\Delta((U'\cup \overline{\beta D'})^{\leftarrow})}\M_c(CS_{x-2,y-2}(U'^{\leftarrow};\beta D'^{\leftarrow};B^{\leftarrow}))}\notag\\
&+\frac{\frac{\Delta(U\cup \overline{\beta D})}{\Delta(U'\cup \overline{\beta D'})}\M_c(CS_{x,y-2}(U'; \beta D';B))\frac{\Delta((U\cup \overline{1D})^{\leftarrow})}{\Delta((U'\cup \overline{1D'})^{\leftarrow})}\M_c(CS_{x-2,y}(U'^{\leftarrow}; D'^{\leftarrow};B^{\leftarrow}))}{\frac{\Delta(U\cup \overline{D})}{\Delta(U'\cup \overline{D'})}\M_c(CS_{x,y}(U',D';B))\frac{\Delta((U\cup \overline{\beta D})^{\leftarrow})}{\Delta((U'\cup \overline{\beta D'})^{\leftarrow})}\M_c(CS_{x-2,y-2}(U'^{\leftarrow};\beta D'^{\leftarrow};B^{\leftarrow}))}=1.
\end{align}
Here, by Figure \ref{SymmetricFern2}(b) and (h),  we have the $U$-index sets of the regions $CS_{x-2,y-2}(U^{\leftarrow};\beta D^{\leftarrow};B^{\leftarrow})$ and $CS_{x-2,y}(U^{\leftarrow}; D^{\leftarrow};B^{\leftarrow})$ are respectively $(U\cup \overline{\beta D})^{\leftarrow}$ and $(U\cup \overline{\beta D})^{\leftarrow}$.

We have a claim
\begin{claim}\label{clm2}
\begin{equation}\label{ratio4}
\dfrac{\Delta(U\cup \overline{1D})}{\Delta(U'\cup \overline{1D'})}\dfrac{\Delta(1U\cup \overline{\beta D})}{\Delta(1U'\cup \overline{\beta D'})}=\dfrac{\Delta(U\cup \overline{D})}{\Delta(U'\cup \overline{D'})}\dfrac{\Delta(U^{\leftarrow}\cup \overline{\beta D^{\leftarrow}})}{\Delta(U'^{\leftarrow}\cup \overline{\beta D'^{\leftarrow}})},
\end{equation}
\begin{equation}\label{ratio5}
\dfrac{\Delta(1U\cup \overline{D})}{\Delta(1U'\cup \overline{D'})}\dfrac{\Delta(U\cup \overline{1\beta D})}{\Delta(U'\cup \overline{1\beta D'})}=\dfrac{\Delta(U\cup \overline{D})}{\Delta(U'\cup \overline{D'})}\dfrac{\Delta(U^{\leftarrow}\cup \overline{\beta D^{\leftarrow}})}{\Delta(U'^{\leftarrow}\cup \overline{\beta D'^{\leftarrow}})},
\end{equation}
and
\begin{equation}\label{ratio6}
\dfrac{\Delta(U\cup \overline{\beta D})}{\Delta(U'\cup \overline{\beta D'})}\dfrac{\Delta(U\cup \overline{1D})}{\Delta(U'\cup \overline{1D'})}=\dfrac{\Delta(U\cup \overline{D})}{\Delta(U'\cup \overline{D'})}\dfrac{\Delta(U^{\leftarrow}\cup \overline{\beta D^{\leftarrow}})}{\Delta(U'^{\leftarrow}\cup \overline{\beta D'^{\leftarrow}})}.
\end{equation}
\end{claim}

\begin{proof}[Proof of Claim \ref{clm2}]
We verify (\ref{ratio4}) first. Eliminate shifting operation and rewrite (\ref{ratio4}) as
\begin{equation}\label{ratio4b}
\dfrac{\Delta(U\cup \overline{1D})\Delta(1U\cup \overline{\beta D})}{\Delta(U\cup \overline{D})\Delta(U\cup \overline{\beta D})}=\frac{\Delta(U'\cup \overline{1D'})\Delta(1U'\cup \overline{\beta D'})}{\Delta(U'\cup \overline{D'})\Delta(U'\cup \overline{\beta D'})}
\end{equation}
The fraction on the left-hand side can be simplified as
\begin{align}
&\dfrac{\Delta(U\cup \overline{1D})\Delta(1U\cup \overline{\beta D})}{\Delta(U\cup \overline{D})\Delta(U\cup \overline{\beta D})}\notag\\
&=|1- \overline{\beta}|\prod_i|\overline{1}-s_i|\prod_j|\overline{1}-\overline{t_j}| \prod_{i}|1-s_i|\prod_j|1-\overline{t_j}|.
\end{align}
Doing similarly with the second fraction, we have the right-hand side of (\ref{ratio4b}) becomes
\begin{align}
|1- \overline{\beta}|\prod_i|\overline{1}-s'_i|\prod_j|\overline{1}-\overline{t'_j}| \prod_{i}|1-s'_i|\prod_j|1-\overline{t'_j}|
\end{align}
As the reflections do not  change the difference between indices, we have left- and right-hand sides of (\ref{ratio4b}) are respectively equal to
\begin{equation}
|1- \overline{\beta}|\prod_i|\overline{1}-s_i|\prod_j|\overline{1}-t_j|
\prod_j|1-t_j| \prod_{i}|1-s_i|
\end{equation}
and
\begin{equation}
|1- \overline{\beta}|\prod_i|\overline{1}-s'_i|\prod_j|\overline{1}-t'_j|\prod_j|1-t'_j| \prod_{i}|1-s'_i|.
\end{equation}
Moreover, since $U\cup D=U'\cup D'$ and $U\cap D=U'\cap D'$,  we have
\[\prod_i|\overline{1}-s_i|\prod_j|\overline{1}-t_j|=\prod_i|\overline{1}-s'_i|\prod_j|\overline{1}-t'_j|\]
and
\[\prod_j|1-t_j| \prod_{i}|1-s_i|=\prod_j|1-t'_j| \prod_{i}|1-s'_i|.\]
This implies that two sides of (\ref{ratio4b}) are equal, and so does (\ref{ratio4}).

Identities (\ref{ratio5}) and (\ref{ratio6}) can be verified in a completely analogous manner.
\end{proof}

By (\ref{ratio4})--(\ref{ratio6}),  (\ref{recurrence2c}) is equivalent to
\begin{align}\label{recurrence2d}
&\frac{\M_c(CS_{x,y-2}(U'; 1D';B))\M_c(CS_{x-2,y-2}(1U'; \beta D';B))}{\M_c(CS_{x,y}(U',D';B))\M_c(CS_{x-2,y-2}(U'^{\leftarrow};\beta D'^{\leftarrow};B^{\leftarrow}))}+
\frac{\M_c(CS_{x,y-2}(1U'; D';B))\M_c(CS_{x-2,y-2}(U'; 1\beta D';B))}{\M_c(CS_{x,y}(U';D';B))\M_c(CS_{x-2,y-2}(U'^{\leftarrow};\beta D'^{\leftarrow};B^{\leftarrow}))}\notag\\
&+\frac{\M_c(CS_{x,y-2}(U'; \beta D';B))\M_c(CS_{x-2,y}(U'; D';B))}{\M_c(CS_{x,y}(U';D';B))\M_c(CS_{x-2,y-2}(U'^{\leftarrow};\beta D'^{\leftarrow};B^{\leftarrow}))}=1,
\end{align}
or equivalently,
{\small\begin{align}\label{recurrence2e}
&\M_c(CS_{x,y}(U';D';B))\M_c(CS_{x-2,y-2}(U'^{\leftarrow};\beta D'^{\leftarrow};B^{\leftarrow}))=\M_c(CS_{x,y-2}(U'; 1D';B))\M_c(CS_{x-2,y-2}(1U'; \beta D';B))\notag\\
&+\M_c(CS_{x,y-2}(1U'; D';B))\M_c(CS_{x-2,y-2}(U'; 1\beta D';B))+\M_c(CS_{x,y-2}(U'; \beta D';B))\M_c(CS_{x-2,y}(U'^{\leftarrow}; D'^{\leftarrow};B^{\leftarrow})).
\end{align}}
which follows from the application of recurrence (\ref{recurrence2}) to the region $CS_{x,y}(U';D';B)$. This finishes our verification that the right-hand side of (\ref{maineqrefine}) satisfies  recurrence (\ref{recurrence2}), and its finishes the proof in the case $1\notin U\cup D$.

\bigskip

\textbf{Case 4. $1\in U \cap D$.}

\bigskip

\begin{figure}\centering
\includegraphics[width=11cm]{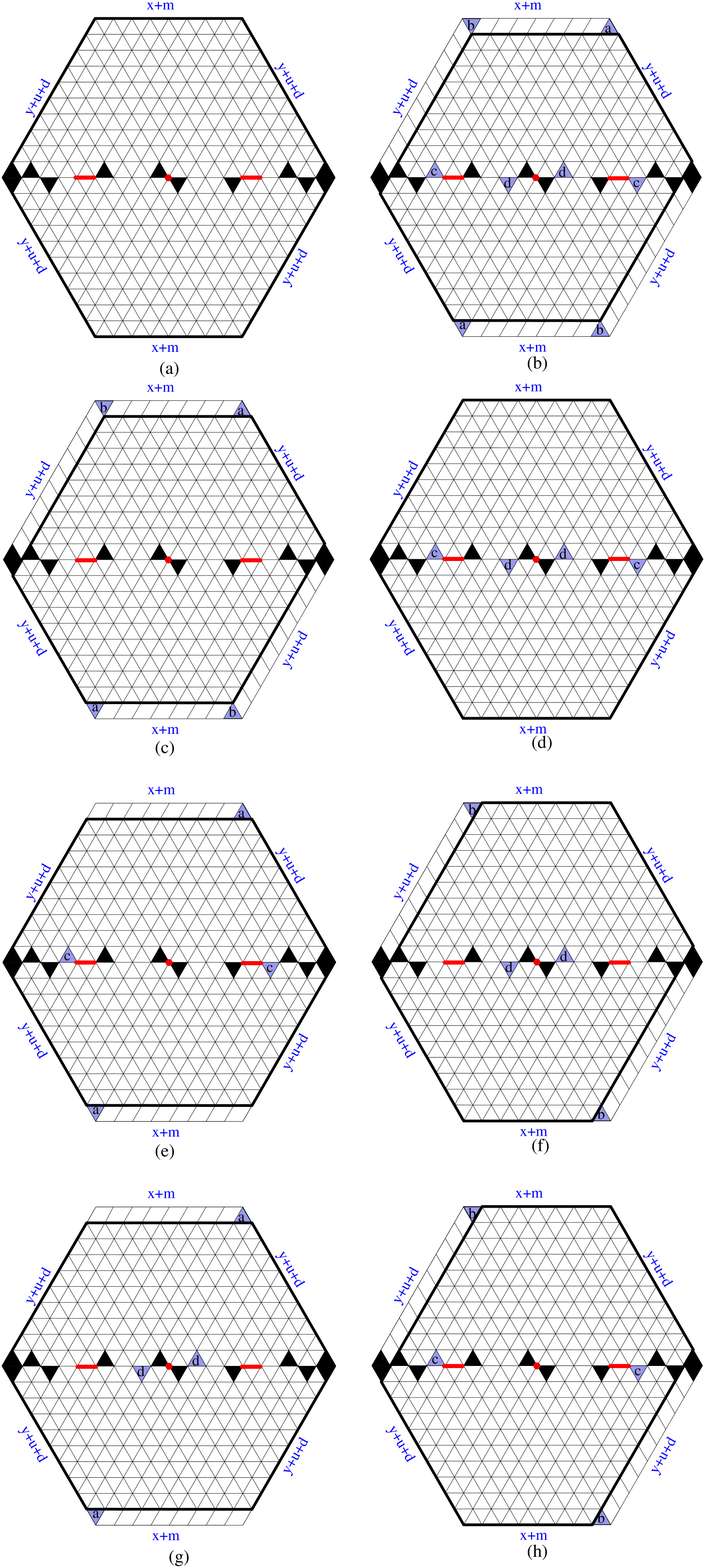}
\caption{How to apply Kuo condensation in the case when $1\in U\cap D$.}\label{SymmetricFern3}
\end{figure}

\begin{figure}\centering
\includegraphics[width=11cm]{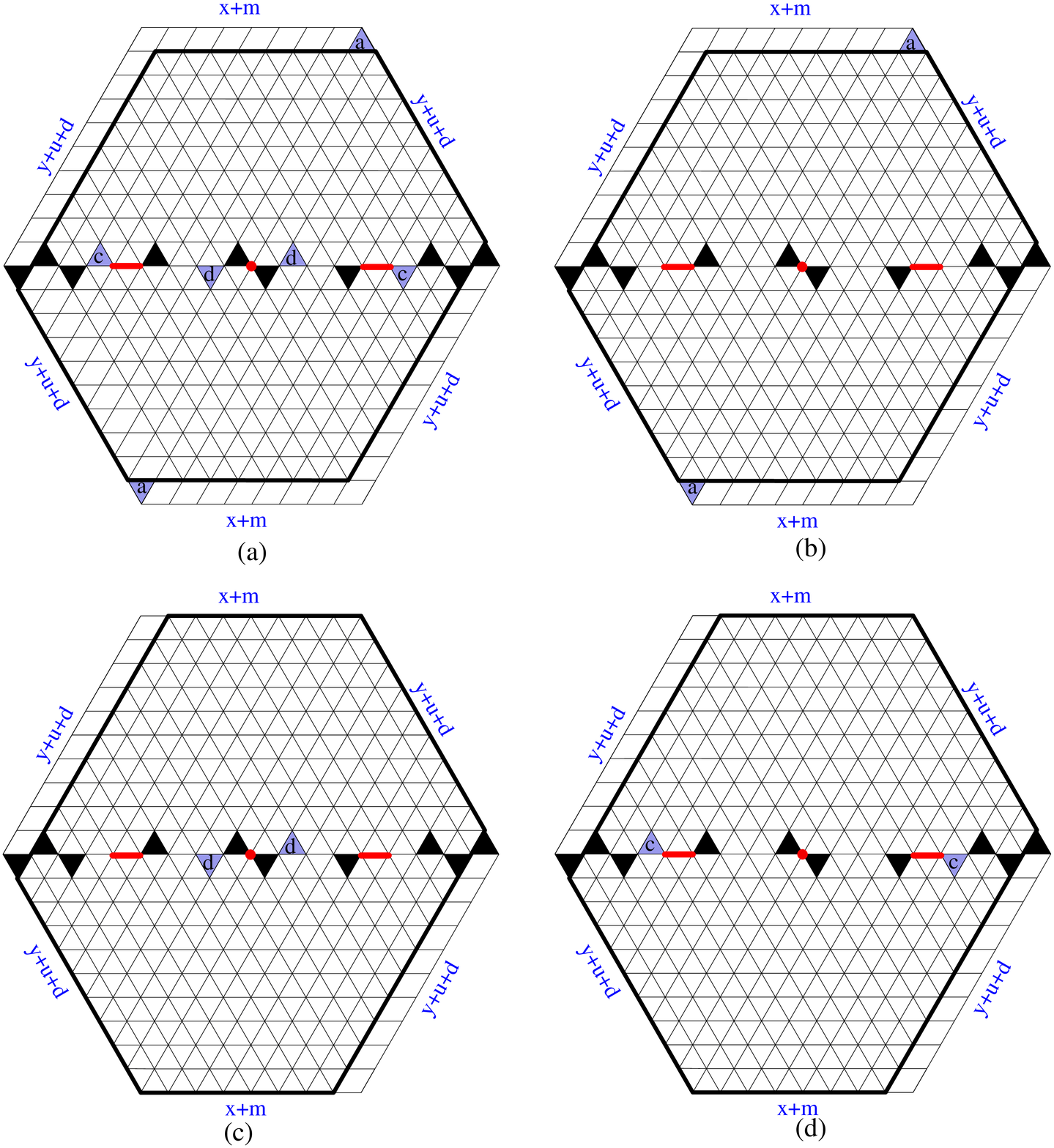}
\caption{Reforming the leftover regions in the case of the regions corresponding to (a) $G_{abcd}$, (b) $G_{ab}$, (c) $G_{bd}$, and (d) $G_{bc}$.}\label{SymmetricFern3b}
\end{figure}

We also apply Kuo condensation in Theorem \ref{ciukuo1} with the choice of the eight vertices as shown in Figure \ref{SymmetricFern3}(b). Consider the region corresponding with the graph $G_{abcd}$. The removal of the unit triangles corresponding to $a_1,a_2,b_1,b_2$ yields forced lozenges along the north, southeast, south, and northwest sides of the region (see Figure \ref{SymmetricFern3}(b)). After removing the forced lozenges from the region, we get the region restricted by the bold contour. This leftover region is not yet a doubly-dented hexagon, however, as shown in Figure \ref{SymmetricFern3b}(a), it has the same tiling number as the region $CS_{x-2,y-2}(\alpha U\setminus \{1\};\beta D;B)$. This means that
\begin{equation}\label{eq3b}
\M_c(G_{abcd})=\M_c(CS_{x-2,y-2}(\alpha U\setminus \{1\};\beta D;B)).
\end{equation}
Working similarly on the regions corresponding with the others terms in the recurrence (\ref{kuoeq}) as shown in Figures \ref{SymmetricFern1}(c)--(h), we have
\begin{equation}\label{eq3c}
\M_c(G_{ab})=\M_c(CS_{x,y}(U\setminus \{1\}; D;B)),
\end{equation}
\begin{equation}\label{eq3d}
\M_c(G_{cd})=\M_c(CS_{x-2,y-2}(\alpha U; \beta D;B)),
\end{equation}
\begin{equation}\label{eq3e}
\M_c(G_{ac})=\M_c(CS_{x,y-2}(\alpha U; D;B)),
\end{equation}
\begin{equation}\label{eq3f}
\M_c(G_{bd})=\M_c(CS_{x-2,y}(U\setminus \{1\}; \beta D;B)),
\end{equation}
\begin{equation}\label{eq3g}
\M_c(G_{ad})=\M_c(CS_{x,y-2}(U; \beta D;B)),
\end{equation}
\begin{equation}\label{eq3h}
\M_c(G_{bc})=\M_c(CS_{x-2,y}(\alpha U\setminus \{1\};D;B)).
\end{equation}
Strictly speaking, similar to the case of the graph $G_{abcd}$, the leftover regions after removing forced lozenges from the regions corresponding to $G_{ab}$, $G_{bd}$, and $G_{bc}$ shown in Figures \ref{SymmetricFern3}(c), (f) and (h) are \emph{not} doubly--dented regions. However, they have the same tiling numbers as the regions
$CS_{x,y}(U\setminus \{1\}; D;B)$, $CS_{x-2,y}(U\setminus \{1\} ; \beta D;B)$ and $CS_{x-2,y}(\alpha U\setminus \{1\};D;B)$, respectively. Indeed, as shown in Figures \ref{SymmetricFern3b}(b), (c), (d), the latter three regions, after removing forced lozenges, are congruent with the leftover regions in  Figures \ref{SymmetricFern3}(c), (f) and (h). This yields identities (\ref{eq2c}), (\ref{eq2f}) and (\ref{eq2h}) above.

Therefore, we have the recurrence:
\begin{align}\label{recurrence3}
&\M_c(CS_{x,y}(U;D;B))\M_c(CS_{x-2,y-2}(\alpha U\setminus\{ 1\};\beta D;B))=\M_c(CS_{x,y}(U\setminus \{1\}; D;B))\M_c(CS_{x-2,y-2}(\alpha U; \beta D;B))\notag\\
&+\M_c(CS_{x,y-2}(\alpha U; D;B))\M_c(CS_{x-2,y}(U\setminus \{1\} ; \beta D;B))+\M_c(CS_{x,y-2}(U; \beta D;B))\M_c(CS_{x-2,y}(\alpha U\setminus \{1\};D;B)).
\end{align}

\bigskip

To complete the proof in this case, we need to verify that the right-hand side of (\ref{maineqrefine}) satisfy above recurrence  (\ref{recurrence3}). Equivalently, we need to verify:
\begin{align}\label{recurrence3c}
&\frac{\frac{\Delta(U\setminus \{1\}\cup \overline{D})}{\Delta(U'\setminus  \{1\}\cup \overline{D'})}\M_c(CS_{x,y}(U'\setminus  \{1\}; D')) \frac{\Delta(\alpha U\cup \overline{\beta D})}{\Delta(\alpha U'\cup \overline{\beta D'})}\M_c(CS_{x-2,y-2}(\alpha U'; \beta D';B))}{\frac{\Delta(U\cup \overline{D})}{\Delta(U'\cup \overline{D'})}\M_c(CS_{x,y}(U';D';B)) \frac{\Delta(\alpha U\setminus \{1\}\cup \overline{\beta D})}{\Delta(\alpha U'\setminus  \{1\}\cup \overline{\beta D'})}\M_c(CS_{x-2,y-2}(\alpha U'\setminus  \{1\};\beta D';B))}\notag\\
&+\frac{\frac{\Delta(\alpha U\cup \overline{D})}{\Delta(\alpha U'\cup \overline{D'})}\M_c(CS_{x,y-2}(\alpha U'; D';B))\frac{\Delta(U\setminus \{1\}\cup \overline{\beta D})}{\Delta(U'\setminus  \{1\}1\cup \overline{\beta D'})}\M_c(CS_{x-2,y}(U'\setminus \{1\} ; \beta D';B))}{\frac{\Delta(U\cup \overline{D})}{\Delta(U'\cup \overline{D'})}\M_c(CS_{x,y}(U';D';B)) \frac{\Delta(\alpha U\setminus \{1\}\cup \overline{\beta D})}{\Delta(\alpha U'\setminus  \{1\}\cup \overline{\beta D'})}\M_c(CS_{x-2,y-2}(\alpha U'\setminus  \{1\};\beta D';B))}\notag\\
&+\frac{\frac{\Delta(U\cup \overline{\beta D})}{\Delta(U'\cup \overline{\beta D'})}\M_c(CS_{x,y-2}(U; \beta D';B)) \frac{\Delta(\alpha U\setminus \{1\}\cup \overline{D})}{\Delta(\alpha U'\setminus  \{1\}\cup \overline{D'})} \M_c(CS_{x-2,y}(\alpha U\setminus \{1\};D';B))}{\frac{\Delta(U\cup \overline{D})}{\Delta(U'\cup \overline{D'})}\M_c(CS_{x,y}(U';D';B)) \frac{\Delta(\alpha U\setminus \{1\}\cup \overline{\beta D})}{\Delta(\alpha U'\setminus  \{1\}\cup \overline{\beta D'})}\M_c(CS_{x-2,y-2}(\alpha U'\setminus  \{1\};\beta D';B))}=1.
\end{align}

Similar to the cases treated above, we would like to prove the following three identities:
\begin{claim}\label{clm3}
\begin{equation}\label{ratio7}
\dfrac{\Delta(U\setminus \{1\}\cup \overline{D})}{\Delta(U'\setminus \{1\}\cup \overline{D'})}\dfrac{\Delta(\alpha U\cup \overline{\beta D})}{\Delta(\alpha U'\cup \overline{\beta D'})}=\dfrac{\Delta(U\cup \overline{D})}{\Delta(U'\cup \overline{D'})} \dfrac{\Delta(\alpha U\setminus \{1\}\cup \overline{\beta D})}{\Delta(\alpha U'\setminus \{1\}\cup \overline{\beta D'})},
\end{equation}
\begin{equation}\label{ratio8}
\dfrac{\Delta(\alpha U\cup \overline{D})}{\Delta(\alpha U'\cup \overline{D'})}\dfrac{\Delta(U\setminus \{1\}\cup \overline{\beta D})}{\Delta(U'\setminus \{1\}\cup \overline{\beta D'})}=\dfrac{\Delta(U\cup \overline{D})}{\Delta(U'\cup \overline{D'})} \dfrac{\Delta(\alpha U\setminus \{1\}\cup \overline{\beta D})}{\Delta(\alpha U'\setminus \{1\}\cup \overline{\beta D'})},
\end{equation}
and
\begin{equation}\label{ratio9}
\dfrac{\Delta(U\cup \overline{\beta D})}{\Delta(U'\cup \overline{\beta D'})} \dfrac{\Delta(\alpha U\setminus \{1\}\cup \overline{D})}{\Delta(\alpha U'\setminus \{1\}\cup \overline{D'})}=\dfrac{\Delta(U\cup \overline{D})}{\Delta(U'\cup \overline{D'})} \dfrac{\Delta(\alpha U\setminus \{1\}\cup \overline{\beta D})}{\Delta(\alpha U'\setminus \{1\}\cup \overline{\beta D'})}.
\end{equation}
\end{claim}
Claim \ref{clm3} can be proved similarly to Claim \ref{clm1} and is omitted.
\medskip

By (\ref{ratio7})--(\ref{ratio9}),  we have (\ref{recurrence3c}) rewritten as
{\small \begin{align}\label{recurrence3d}
&\frac{\M_c(CS_{x,y}(U'\setminus \{1\}; D'))\M_c(CS_{x-2,y-2}(\alpha U'; \beta D';B))}{\M_c(CS_{x,y}(U';D';B))\M_c(CS_{x-2,y-2}(\alpha U'\setminus \{1\};\beta D';B))}
+\frac{\M_c(CS_{x,y-2}(\alpha U'; D';B))\M_c(CS_{x-2,y}(U'\setminus\{1\} ; \beta D';B))}{\M_c(CS_{x,y}(U';D';B))\M_c(CS_{x-2,y-2}(\alpha U'\setminus \{1\};\beta D';B))}\notag\\
&+\frac{\M_c(CS_{x,y-2}(U; \beta D';B))\M_c(CS_{x-2,y}(\alpha U\setminus \{1\};D';B))}{\M_c(CS_{x,y}(U';D';B))\M_c(CS_{x-2,y-2}(\alpha U'\setminus \{1\};\beta D';B))}=1,
\end{align}}
equivalently,
{\small \begin{align}\label{recurrence3e}
&\M_c(CS_{x,y}(U';D';B))\M_c(CS_{x-2,y-2}(\alpha U'\setminus \{1\};\beta D';B))=\M_c(CS_{x,y}(U'\setminus \{1\}; D';B))\M_c(CS_{x-2,y-2}(\alpha U'; \beta D';B))\notag\\
&+\M_c(CS_{x,y-2}(\alpha U'; D';B))\M_c(CS_{x-2,y}(U'\setminus \{1\} ; \beta D';B))+\M_c(CS_{x,y-2}(U'; \beta D';B))\M_c(CS_{x-2,y}(\alpha U'\setminus \{1\};D';B)).
\end{align}}
However, this is obtained by applying recurrence (\ref{recurrence3}) to the region $CS_{x,y}(U';D';B)$. This finishes our proof for the case when $x+y$ is even.

\bigskip

\begin{figure}\centering
\includegraphics[width=10cm]{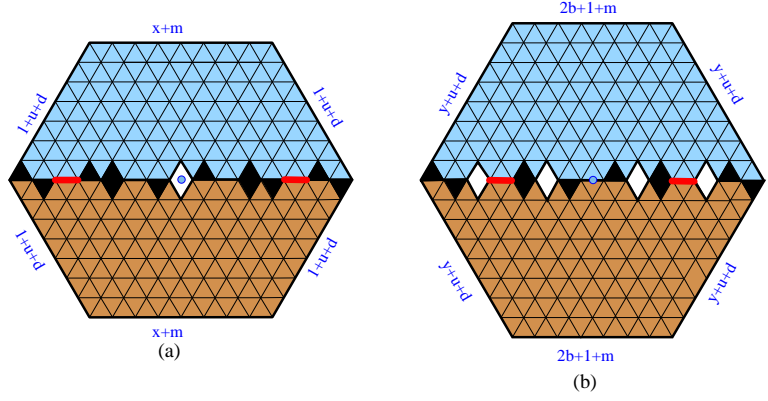}
\caption{Two base cases when $x+y$ is odd: (a) $y=1$, (b) $x=2b+1$.}\label{basecase2}
\end{figure}

Next, we consider the case when $x+y$ is odd, i.e. $x$ and $y$ have opposite parities. This case can be handled similarly to the case when $x+y$ is even.
We also prove (\ref{maineqrefine}) by induction on $x+y+u+d$. The base cases are still the cases: $x<2b+2$, $y<2$, and $u+d=0$.

\begin{figure}\centering
\includegraphics[width=11cm]{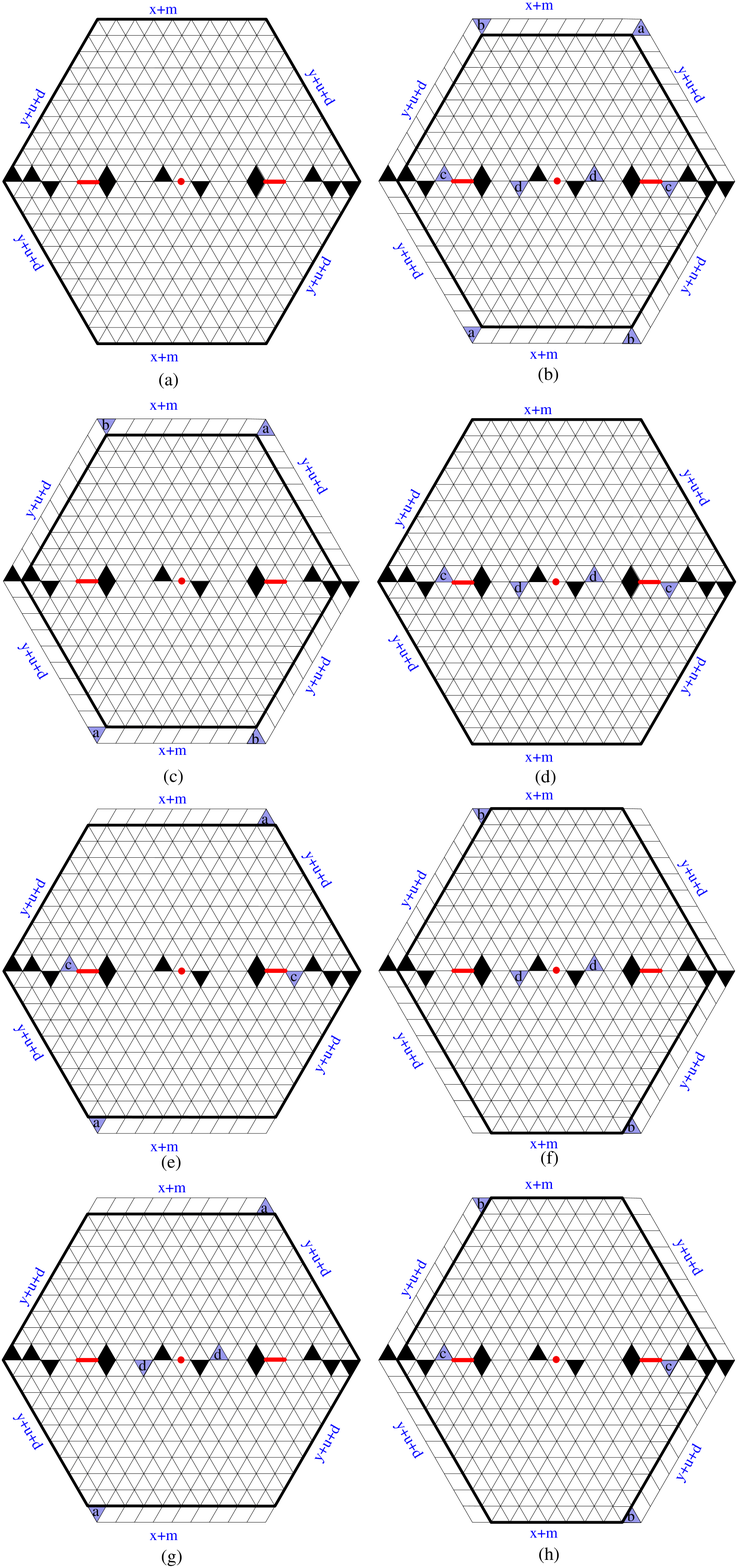}
\caption{How to apply Kuo condensation when $x+y$ is odd.}\label{SymmetricFern4}
\end{figure}

The case $u+d=0$ is still trivial. However,  the case $y<2$ is now divided further into two subcases $y=0$ and $y=1$ (as $y$ now can take odd values). While the case $y=0$ is completely the same as the case treated above when $x+y$ is even, the case $y=1$ is slightly different. When $y=1$, any centrally symmetric tilings of the region must contain a vertical lozenge at the central position (i.e. position $\frac{x+y+2n+1}{2}$). Removing this special forced lozenge, each symmetric tiling can be partitioned into tilings of two congruent dented semihexagons that are images of each other through the central symmetry. It means that
\[M_c(CS_{x,1}(U;D;B))=\M\left(T_{x+m,1+u+d}\left(U\cup \overline{D} \cup\left\{ \frac{x+y+2n+1}{2}\right\}\right)\right).\]
Similarly, we have
\[M_c(CS_{x,1}(U';D';B))=\M\left(T_{x+m,1+u+d}\left(U'\cup \overline{D'} \cup\left\{ \frac{x+y+2n+1}{2}\right\}\right)\right),\]
and (\ref{maineqrefine}) follows directly from Cohn--Larsen--Propp's formula (\ref{CLPeq}).

We next consider the case when $x<2b+2$.  As $x\geq 2b$, $x$ may be $2b$ or $2b+1$ (i.e. $x$ can take an odd value now). The case $x=2b$ is still the same as in the base case $x=2b$ when $x+y$ is even. We now consider the case $x=2b+1$. It is easy to see that, each centrally symmetric tiling of the region contain a vertical lozenge at any positions in $\mathcal{O}^c$, except for the central position. This means that each centrally symmetric tiling of the region $CS_{2b+1,y}(U;D;B)$ can be partitioned into tilings of two congruent dented semihexagons that are images of each other through  the central symmetry (see Figure \ref{basecase2}(b)). In particular, we have
\[\M_c(CS_{2b+1,y}(U;D;B))=\M\left(T_{2b+1+m,y+u+d}\left(\left(\mathcal{O}^c\setminus \left\{\frac{x+y+2n+1}{2}\right\}\right)\cup U\cup \overline{D}\right)\right).\]
Note that, as mentioned above, the middle position must be not in $U\cup D\cup B$ in order for the region has a centrally symmetric tiling. This guarantees that the middle position is always a dent of the dented semihexagon on the right-hand side. Similarly, we have
\[\M_c(CS_{2b+1,y}(U';D';B))=\M\left(T_{2b+1+m,y+u+d}\left(\left(\mathcal{O}^c\setminus \left\{\frac{x+y+2n+1}{2}\right\}\right)\cup U'\cup \overline{D'}\right)\right),\]
 and (\ref{maineqrefine}) follows again from Cohn--Larsen--Propp's formula.

For the induction step, we process exactly the same as that in the case when $x+y$ is even, the only difference is that when the  central position is not in $U$ (and by symmetry, it is also not in $D$), we apply Kuo condensation in Theorem \ref{ciukuo2} (instead of applying Theorem \ref{ciukuo1}) as shown in Figure \ref{SymmetricFern4}, for the case $1\in U\setminus D$; the other cases are similar. We still have the left-hand side of (\ref{maineqrefine}) satisfies the same recurrences as (\ref{recurrence1}), (\ref{recurrence2}), and (\ref{recurrence3}) for the cases $1\in U\setminus D$ (and its $180^{\circ}$ rotation), $1\notin U\cup D$, and $1\in U \cap D$, respectively. The verification that the right-hand side of (\ref{maineqrefine}) satisfies the same recurrences are completely the same as in the case of even $x+y$ treated above.



\begin{thebibliography}{50}
\bibitem{Andrews} G. E. Andrews, Plane partitions (III): The weak Macdonald conjecture, \emph{Invent. Math.}, \textbf{53} (1979), 193--225.




\bibitem{Ciu1} M. Ciucu, Another dual of MacMahon's theorem on plane partitions, \emph{Adv. Math.}, \textbf{306} (2017),  427--450.

\bibitem{Ciu3} M. Ciucu, Symmetries of Shamrocks IV: The Self-Complementary Case. \emph{Preprint: \texttt{arXiv:1906.02022}}.

\bibitem{Ciu4} M. Ciucu, Centrally symmetric tilings of fern-cored hexagons. \emph{Preprint: \texttt{arXiv:1906.02951}}.








\bibitem{Twofern} M. Ciucu and T. Lai, Lozenge tilings doubly-intruded hexagons, \emph{J. Combin. Theory Ser. A}, 167 (2019): 294--339.  

\bibitem{CLP} H. Cohn, M. Larsen, and J. Propp, The shape of a typical boxed plane partition,
\emph{New York J. Math.}, \textbf{4} (1998), 137--165.






\bibitem{Zeil} C. Koutschan, M. Kauer, and Zeilberger, A proof of George Andrews' and David Robbins' $a$ TSPP-conjecture. \emph{Proc. Natl. Acad. Sci. USA}, \textbf{108} (2011), 2196--2199.

\bibitem{Krat3}
C. Krattenthaler, Plane partitions in the work of Richard Stanley and his school, ``\emph{The Mathematical Legacy of Richard P. Stanley}" P. Hersh, T. Lam, P. Pylyavskyy and V. Reiner (eds.), Amer. Math. Soc., R.I., 2016, pp. 246-277.

\bibitem{Kuo} E. H. Kuo, Applications of graphical condensation for enumerating matchings and tilings,
\emph{Theor. Comput. Sci.}, \textbf{319} (2004), 29--57.


\bibitem{Kup} G. Kuperberg, Symmetries of plane partitions and the permanent-determinant method, \emph{J. Combin. Theory Ser. A}, \textbf{68} (1994), 115--151.




\bibitem{Tri2} T. Lai, A $q$-enumeration of lozenge tilings of a hexagon with three dents, \emph{Adv.  Applied Math.}, \textbf{82} (2017), 23--57.

\bibitem{Tri1} T. Lai, A $q$-enumeration of a hexagon with four adjacent triangles removed from the boundary, \emph{European J.  Combin.}, \textbf{64} (2017), 66--87.

\bibitem{Halfhex1}
T. Lai, Lozenge Tilings of a Halved Hexagon with an Array of Triangles Removed from the Boundary, \emph{SIAM J. Discrete Math.}, \textbf{32}(1) (2018), 783--814.

\bibitem{Halfhex2}
T. Lai, Lozenge Tilings of a Halved Hexagon with an Array of Triangles Removed from the Boundary, Part II, \emph{Electron. J.  Combin.}, \textbf{25}(4) (2018), \#P4.58.

\bibitem{Halfhex3}
T. Lai, Tiling Enumeration of Doubly-intruded Halved Hexagons , \emph{Preprint: \texttt{arXiv:1801.00249}}.

\bibitem{Threefern}
T. Lai, Lozenge Tilings of Hexagons with Central Holes and Dents, \emph{Preprint: \texttt{arXiv:1803.02792}}.

\bibitem{Threefern2}
T. Lai, Tiling Enumeration of Hexagons with Off-central Holes, \emph{Preprint: \texttt{arXiv:1905.07119}}.

\bibitem{Minor}
T. Lai, Proof of a Conjecture of Kenyon and Wilson on Semicontiguous Minors, \emph{J. Combin. Theory Ser. A}, \textbf{116} (2019), 134--163.



\bibitem{shuffling} T.~Lai and R.~Rohatgi, A Shuffling Theorem for Lozenge Tilings of Doubly-Dented Hexagons, \emph{Preprint: \texttt{arXiv:1905.08311}}.

\bibitem{shuffling2} T.~Lai, A Shuffling Theorem for Reflectively Symmetric Lozenge Tilings, \emph{Preprint: \texttt{arXiv:1905.09268}}.

\bibitem{Mac} P. A. MacMahon, Memoir on  the theory of the partition of numbers---Part V. Partition in two-dimensional space, \emph{Phil. Trans. R. S.}, 1911, A.







\bibitem{Ranjan1}
R. Rohatgi, Enumeration of lozenge tilings of halved hexagons with a boundary defect, \emph{Electron. J. Combin.},  \textbf{22}(4) (2015), P4.22.

\bibitem{Ranjan2}
R. Rohatgi, Enumeration of tilings of a hexagon with a maximal staircase and a unit triangle removed, \emph{Australas. J. Combin.}, \textbf{65}(3) (2016), 220--231.



\bibitem{Stanley}
R. Stanley, Symmetries of plane partitions, \emph{J. Combin. Theory Ser. A}, \textbf{43} (1986), 103--243.

\bibitem{Stem} J. R. Stembridge, The enumeration of totally symmetric plane partitions, \emph{Adv. in Math.}, \textbf{111} (1995), 227--243.








\end{thebibliography}
\end{document}